\documentclass{amsart}
\usepackage{amssymb}
\usepackage{graphicx}
\def\R{{\mathbb R}}
\def\N{{\mathbb N}}

\def\Z{{\mathbb Z}}
\def\1{{1\!\!1}}
\def\E{{\mathbb E}}
\def\P{{\mathbb P}}

\def\cal{\mathcal}
\def\eps{\varepsilon}

\makeatletter

\@addtoreset{equation}{section}
\title[Explicit Lyapunov functions  of Jackson networks]{Explicit Lyapunov functions and
  estimates of the essential spectral radius for Jackson networks}
\author[I. Ignatiouk-Robert]{Irina Ignatiouk-Robert}
\address[I. Ignatiouk-Robert]{
{Universit\'e de Cergy-Pontoise,}
{AGM, D\'epartement de math\'ematiques,}
{UMR 8088,} 
{2, Avenue Adolphe Chauvin,}
{95302 Cergy-Pontoise Cedex,}
{France}}
\email[I. Ignatiouk]{Irina.Ignatiouk@u-cergy.fr}
\author[D. Tibi]{Danielle Tibi}
\address[D. Tibi]{
{Universit\'e Paris-7,}
{LPMA, UMR 7599}
{Universit\'e Paris-Didrot,}
{Site Chevaleret, Case 7012,}
{75205, Paris, Cedex 13}
{France}}
\email[D.Tibi]{tibi@math.univ-paris-diderot.fr}
\date{\today}
\keywords{Jackson networks. Twisted processes. Lyapunov function. Stationary
  distribution. Essential spectral radius}
\subjclass{60J27, 37B25, 60K25, 90B15}

\newtheorem{theorem}{Theorem}
\newtheorem{prop}{Proposition}[section]
\newtheorem{cor}{Corollary}[section]
\newtheorem{defi}{Definition}
\newtheorem{lemma}{Lemma}[section]

\newcommand{\be}{\begin{equation}}
\newcommand{\ee}{\end{equation}}

\begin{document}
\begin{abstract}
A family of explicit Lyapunov functions for positive recurrent Markovian Jackson networks is constructed. With this result we obtain explicit estimates of the tail distribution of the first time when the process returns to  large compact sets and some explicit estimates of the essential radius of the process. The  essential spectral radius  of the process provides the best  geometric convergence rate to  equilibrium that one can get by changing the transitions of the process in a finite set. 
\end{abstract}
\maketitle 

\section{Introduction}\label{sec1}

Before formulating our results we recall the definition and some well known results
concerning classical Jackson networks, see~\cite{Kelly} for example. For a Jackson network with $d$
queues,  the arrivals at the $i$-th queue are Poisson with parameter $\lambda_i$ and the
services delivered by the server are exponentially distributed with parameters
$\mu_i$. All the Poisson processes and the services are assumed to be independent. The routing matrix is
denoted $P=(p_{ij};\,i,j=1,\ldots,d)$, $p_{ij}$ is the probability that a
customer goes to
the $j$-th queue when he has finished his service at queue $i$. The
residual quantity \[
p_{i0} = 1 - \sum_{j=1}^d p_{ij}
\]
is the probability that this customer leaves definitively the network. Without any
restriction of generality we can assume that $p_{ii}=0$ for all $i\in\{1,\ldots,d\}$. 

Denote by $Z_i(t)$ the length of the queue $i$ at time $t$. Then the process $Z(t)=(Z_1,(t),\ldots,Z_d(t))$ is
a continuous time Markov process on $\Z^d_+$ generated by 
\[ {\cal L}f
(y) = \sum_{z\in \Z^d_+} q(y,z) (f(z) - f(y)), \quad y\in \Z^d_+, 
\]
with $q(y,z)=q(z-y)$ such that 
\begin{equation}\label{eq1-1} q(y) = \left\{
\begin{array}{ll}
\lambda_i, \ &\mbox{ if } y = \epsilon^i, \; i\in \{1,\ldots ,d\}, \\
\mu_ip_{i0}, &\mbox{
if } y = - \epsilon^i, \; i\in \{1,\ldots ,d\}, \\ \mu_ip_{ij}, &\mbox{ if
} y = \epsilon^j
- \epsilon^i,
\; i,j\in \{1,\ldots ,d\},\\ 0, &\mbox{ otherwise,} \end{array} \right.
\end{equation}
where $\epsilon^i$ denotes the $i$th
unit vector, $\epsilon^i_j=0$ if $j\not=i$ and $\epsilon^i_i=1$. It is
convenient to put  $p_{00}
= 1$ and $p_{0i} = 0$ for $i \not= 0$, the matrix
$(p_{ij};\,i,j=0,\ldots,d)$ is then
stochastic. We denote by $p^{(n)}_{ij}$  the
$n$-time transition probabilities of a Markov chain with $d+1$ states
associated to the
stochastic matrix $(p_{ij};\,i,j=0,\ldots,d)$. 
\medskip

{\bf Assumption~(A).} We suppose that the matrix $(q(x-y);\,x,y\in \Z^d)$ is
irreducible.
\medskip

\noindent
This assumption is equivalent to the following conditions
\begin{itemize} \item[$(A_1)$] Every customer
leaves the network with probability $1$, i.e. for any $i\in \{ 1,\ldots
,d\}$ there exists
$n\in \N,$ such that $p^{(n)}_{i0} > 0.$ This condition is satisfied if and only if the
spectral radius of the matrix $(p_{ij};\,i,j=1,\ldots,d)$ is strictly less
than unity.
\item[$(A_2)$] for any $ i= 1,\ldots,d$, there exist $ n\in \N $ and $ j\in \{
1,\ldots,d\} $
such that $\lambda_jp_{ji}^{(n)}>0.$  \end{itemize}

Under the assumption~$(A_1)$, the system of traffic equations  
\begin{equation}\label{eq1-2} \nu_j = \lambda_j + \sum_{i=1}^d
\nu_i p_{ij}, \; \; j=1,\ldots ,d. \end{equation}
has a unique solution $(\nu_i)$, and this solution satisfies $\nu_i >0$ for all  $i \in \{1,\ldots,d \}$. The Markov
process
$(Z(t))$ is ergodic (positive recurrent) if and only if
\begin{equation}\label{eq1-4}
\nu_i < \mu_i \; \; \mbox{ for all } \; i=1, \ldots ,d, \end{equation}
and the stationary probabilities $(\pi(x);\,x\in
\Z^d_+)$ are given by the product
formulae
\begin{equation}\label{eq1-3}
\pi (x) = \prod_{i=1}^d {({\nu_i}/{\mu_i})^{x_i}}{(1 - {\nu_i}/{\mu_i})}, \; \; x\in \Z^d_+.
\end{equation}  

\medskip

{\bf Assumption~(B).} We assume that the inequalities \eqref{eq1-4} hold. 

\medskip

Fayolle, Malyshev, Men'shikov and
Sidorenko~\cite{F-M-M-S} proved that the rate of convergence to stationary distribution
for ergodic Jackson networks is exponential. The proof of this result relies on the
construction of a positive Lipschitz continuous function $f :\R^d_+\to [0,+\infty[$ satisfying the inequality 
\be\label{eq1-5}
{\cal L}f (x) ~\leq~ - \varepsilon, \quad \forall x\in\Z_+^d\setminus E
\ee
for some $\varepsilon > 0$ and some finite subset $E\subset\Z_+^d$. 
Such a function $f$ is often called a Lyapunov function for the Markov process $(Z(t))$. Using \eqref{eq1-5}
   one can easily show that for $\sigma > 0$ small enough, the function $h(x)
    = \exp(\sigma f(x))$ satisfies the inequality 
\be\label{eq1-6}
{\cal L}h (x) ~\leq~ - \theta h(x), \quad \forall x\in\Z_+^d\setminus E 
\ee
for some $\theta = \theta(\varepsilon) > 0$. Usually, a function $h:\Z_+^d\to\R_+$
    satisfying the inequality \eqref{eq1-6} and such that 
    \be\label{eq1-6p} 
    c_E(f) ~\stackrel{def}{=}~  \ \inf_{x\in \Z_+^d\setminus E}  f(x)  >  0  
    \ee
    is also called a Lyapunov function for $(Z(t))$. To make a difference with a Lyapunov function
    satisfying the inequality \eqref{eq1-5}, we call such a function $h$  a {\em multiplicative
    Lyapunov function}. For the hitting time $\tau_{E} = \inf\{ t > 0 : Z(t)\in E\}$,
    the inequalities \eqref{eq1-6} and \eqref{eq1-6p}  imply that 
\be\label{eq1-7}
\P_x(\tau_{E} > t) \leq~  \frac{1}{c_E(f)} \E_x\Bigl( h(Z(t)), \;  \tau_{E} > t\Bigr) \leq~  \frac{1}{c_E(f)}\exp( - \theta t ) h(x), 
\ee
for all $x\in\Z_+^d\setminus E$. An explicit form for  the multiplicative Lyapunov function $h$ and the quantity $\theta$
whould therefore imply explicit estimates for the tail distribution of the hitting time
$\tau_{E}$. Unfortunately,  construction of an
explicit multiplicative Lyapunov function satisfying \eqref{eq1-6} for a given finite set
$E\subset\Z_+^d$ with the best possible
$\theta$ is usually a very difficult problem. In
\cite{F-M-M-S},  the Lyapunov function $f$ itself and the corresponding set $E$ are both 
rather implicit.

In the present paper we construct a class of explicit multiplicative Lyapunov functions  $h:\Z_+^d\to
[1, +\infty[$ with an explicit  
\[
\theta_h  ~\stackrel{def}{=}~ - \limsup_{|x|\to\infty} {\cal L}h (x)/h(x) ~>~ 0. 
\]
For any such a function $h$ and any $0 < \theta < \theta_h$, one could  therefore identify the set $E$ where \eqref{eq1-6} holds
and get an explicit estimate for the tail distribution of the
hitting time $\tau_E$. 

Using the explicit form of the Lyapunov functions  we obtain an explicit estimate for the {\em essential spectral radius} of the process
$(Z(t))$. Recall that the {\em spectral radius} $r^*$ of the process
 $(Z(t))$ is defined as the infimum of all those $r>0$ for which 
\[
\int_0^\infty r^{-t} \P_x(Z(t)=y)  \, dt < +\infty, \;
\forall x,y\in\Z^d_+. 
\]
When the process $(Z(t))$ is recurrent we obviously have $r^*=1$. The {\it essential spectral
  radius} $r^*_e$ of $(Z(t))$ is the infimum of all those $r>0$ for which there 
is a finite set $E\subset\Z^d_+$ such that 
\[
\int_0^\infty r^{-t} \P_x\bigl( Z(t)=y, \;\tau_E > t \bigr)  \, dt < +\infty \quad 
\text{ for all } \quad  x,y\in\Z^d_+ \setminus E. 
\]
For  the recurrent Markov process  $(Z(t))$,  the quantity $r^*_e$ is equal to the infimum
of all those $r>0$ for which there is a finite set $E\subset\Z^d_+$ such that 
\be\label{eq1-8}
 \int_0^\infty r^{-t} \P_x\bigl(\tau_E > t \bigr)  \, dt < +\infty \quad 
\text{ for all } \quad  x\in\Z^d_+ \setminus E
\ee
(see for instance
Proposition~3.6  of ~\cite{Ignatiouk:05}).  Remark that for those $r>0$ for which \eqref{eq1-8} holds, the function 
\[
h_{r,E}(x) ~\stackrel{def}{=}~  \begin{cases} \int_0^\infty r^{-t} \P_x\bigl(\tau_E > t \bigr)  \, dt , &
\text{ for } \quad  x\in\Z^d_+ \setminus E, \\ 
0 & \text{ for } \quad  x\in E, 
\end{cases}
\]
satisfies the inequalities \eqref{eq1-6} and \eqref{eq1-6p} with a given $E$, $\theta = - \log r$ and 
\[
c_E(f) \geq \int_0^\infty r^{-t} e^{- \sum_i (\lambda_i  + \mu_i) t} \; dt ~\geq~ \left(\ln r + \sum_{i=1}^d(\lambda_i + \mu_i)\right)^{-1} 
\]
The last property of the essential spectral radius $r^*_e$ combined with the estimates
\eqref{eq1-7} shows therefore that the quantity $\theta^*_e = - \log r^*_e$ is equal to the supremum of
all $\theta > 0$ for which there exists a multiplicative Lyapunov function $h:\Z_+^d\to \R_+$ satisfying the inequalities
    \eqref{eq1-6} and \eqref{eq1-6p} for some finite subset $E\subset\Z_+^d$. This is also the best $\theta >
    0$ one could expect to have in \eqref{eq1-7}.

 The essential spectral radius is moreover related to the rate of convergence to  equilibrium. To calculate the rate of convergence to  equilibrium, one should  identify the spectral gap of the transition operator, and except for some very particular processes, this is an extremely difficult problem. Explicit estimates of the rate of convergence are therefore of interest. Malyshev  and
Spieksma~\cite{Malyshev-Spieksma}  proved that for some general class of Markov chains,  the quantity $r^*_e$ gives an
accurate bound for that : this is the  best geometric
convergence rate  one can get by changing the transitions  of  the process on finite
subsets of states.   By Perssons principle (see Liming
Wu~\cite{LimingWu}),  for symmetric Markov chains the quantity $r^*_e$ is related to
the $L^2$-essential spectral radius of the corresponding Markov semi-group.  
For more details concerning the relationship between the quantity 
$r^*_e$ and the rate of convergence to equilibrium see 
Liming Wu~\cite{LimingWu}.

In \cite{Ignatiouk:05}, the quantity $r_e^*$ was represented in terms of the sample path large
deviation rate function $I_{[0,T]}(\cdot)$ of the scaled processes $Z_\eps(t) ~=~ \eps Z(t/\eps)$, $t\in[0,T]$. 
Recall that the family of scaled Markov processes $(Z_\eps(t), \; t\in[0,T])$ satisfies the sample path large deviation  
principle (see~\cite{A-D,D-E, Ignatiouk:04, Ignatiouk:01}) with a good rate function  $I_{[0,T]}(\cdot)$.
Corollary~7.1 of  the paper ~\cite{Ignatiouk:05} proves that 
\be\label{eq1-9}
\log r^*_e ~=~ - ~\inf_{\phi ~: ~\phi(0)=\phi(1), \; \phi(t)\not= 0, \, \forall 0 < t < 1}
~I_{[0,1]}(\phi) 
\ee
where the infimum is taken over all absolutely continuous functions $\phi :[0,1]\to\R^d_+$
with $\phi(0)=\phi(1)$ and such that $\phi(t)\not=0$ for all $0< t < 1$. 
For $d \leq 2$, the quantity $r_e^*$ was calculated explicitly : in this case, the infimum
at the right hand side of \eqref{eq1-9} is
achieved at some constant function $\phi(t) \equiv x=(x^1,\ldots,x^d)\in\R^d_+$ with $x^i >
0$ for some $1\leq i\leq d$ and $x^j=0$ for $j\not=i$. For $d=1$, Proposition~7.1 of
\cite{Ignatiouk:05} shows that 
\[
\log r_e^* ~=~ - (\sqrt{\mu_1} - \sqrt{\lambda_1})^2
\]
and  by Proposition~7.2 of \cite{Ignatiouk:05}, for $d=2$, 
\be\label{eq1-10}
\log r_e^* ~=~ - (1- p_{12}p_{21}) \min\{(\sqrt{\mu_1} - \sqrt{\nu_1})^2, (\sqrt{\mu_2} - \sqrt{\nu_2})^2\}.
\ee
Unfortunately, for higher dimensions $d\geq 3$, the variational problem  \eqref{eq1-9}
seems very difficult to resolve. In the present paper, using the explicit Lyapunov
functions,  we obtain explicit estimates for the essential spectral radius 
$r_e^*$ for an arbitrary dimension $d$. The quantity $r_e^*$  is calculated explicitly for
several examples of Jackson networks.

\section{General  results}\label{sec2} 
To formulate our results, we need to introduce some additional notation : 
\[
G ~\stackrel{def}{=}~ (Id - P)^{-1} ~=~ \sum_{n=0}^\infty P^n,
\]
where $P^n$ denotes the $n$-th iterate of the routing matrix $P=(p_{ij}, \;
i,j\in\{1,\ldots,d\})$, and the series converges because, under our assumptions, the spectral radius of the routing matrix
$P$ is strictly less than unity. We  moreover introduce an auxiliary Markov chain $(\xi_n)$
on $\{0,\ldots,d\}$, with an absorbing state $0$ and  
transition probabilities $p_{ij}$ for $i\in\{1,\ldots,d\}$ and $j\in\{0,\ldots,d\}$. For
$j\in\{1,\ldots,d\}$, we consider $\tau_j=\inf\{n\geq 0 : \xi_n=j\}$ with the convention that
$\inf\emptyset = +\infty$, and  we denote 
\[
Q_{ij} ~=~ \P_i(\tau_j < +\infty) \quad \text{ for } i,j\in\{1,\ldots,d\}. 
\]
\subsection{Explicit Lyapunov functions}
For ${\bf \gamma}=(\gamma_1,\ldots,\gamma_d)\in\R_+^d$, we introduce $d$  vectors $\overrightarrow{\gamma_i} =
(\gamma_i^1,\ldots,\gamma_i^d)$,  $i=1, \ldots , d$, with 
\be\label{2-1}
\gamma_i^j~=~ \log\Bigl( 1  + Q_{ji}\gamma_i\Bigr) \text{ for }
   i,j\in\{1,\ldots,d\}.
\ee
$\Gamma$ denotes the set of all vectors 
${\bf \gamma}\in\R_+^d$ for which the following condition is
satisfied:

\begin{defi} $\gamma\in\Gamma$ if and only if  for any $i=1,\ldots,d$ 
and for any non-zero vector $v=(v^1,\ldots,v^d)\in\R_+^d$ with  $v^i = 0$, 
\be\label{2-2}
\overrightarrow{\gamma_i}\cdot v  ~<~  \sup_{1\leq j\leq d} ~\overrightarrow{\gamma_j}\cdot v
\ee
\end{defi}
Here and throughout, $u \cdot v$ denotes  for $u, v\in\R^d$ the usual scalar product in
$\R^d$.  

\noindent
Our first general preliminary result is the following statement. 

\begin{theorem}\label{pr1} Under the hypothesis (A), for any ${\bf \gamma}\in\Gamma$, the function $h_\gamma : \Z_+^d \to\R_+$, defined by 
\be\label{2-3}
h_\gamma(x) ~=~ \sum_{i=1}^d \exp(\overrightarrow{\gamma_i}\cdot x), \quad x\in\Z_+^d, 
\ee
satisfies the equality
\be\label{2-4}
\limsup_{|x|\to\infty} {\cal L}h_\gamma (x)/h_\gamma(x) ~=~ -~ \min_{1\leq i \leq d}
~ \frac{\gamma _i}{G_{ii}} ~\Bigl( \frac{\mu_i}{1 +\gamma_i} -\nu_i \Bigr). 
\ee
\end{theorem}
The proof of this result is given in Section~\ref{sec3}.

\medskip 

Remark that 
\[\gamma _i\Bigl( \frac{\mu_i}{1 +\gamma_i} -\nu_i \Bigr)
 > 0
\]
if and only if 
\[
0 < \gamma_i < \frac{\mu_i}{\nu_i} -1.
\]
If $\gamma \in \Gamma$ and the last inequalities are satisfied for all $i=1,\ldots,d$, then the right hand side of
\eqref{2-4} is negative and consequently, $h_\gamma$ is a multiplicative Lyapunov
function for $(Z(t))$.

In Section~\ref{examples}, we provide an example of a Jackson network with a completely
symmetrical routing matrix, where the set $\Gamma$ has a simple
explicit representation.  Unfortunately, in general, the explicit description of
the set $\Gamma$ is  a difficult  problem and  it is of interest to give another
equivalent representation  of $\Gamma$. This is a subject of our next result. Here and
throughout, ${\cal M}_1$  denotes the set of probability measures on 
    $\{1,\ldots,d\}$ : 
\[
{\cal M}_1 ~=~ \{ \theta  = (\theta^1, \cdots , \theta^d) \in\R_+^d ~:~ \|\theta\|_1 = 1\}
\]
where $\|\theta\|_1 = |\theta ^1|+\cdots +|\theta ^d|$ is the usual $L^1$ norm in $\R^d$.
For two vertors $a=(a^1,\ldots,a^d)$ and $b=(b^1,\ldots,b^d)$ in $\R^d$ we write $a < b$
if $a^k < b^k$ for all $k=1,\ldots,d$.  

\begin{prop}\label{Gamma}  
1) A vector $\gamma=(\gamma_1,\ldots,\gamma_d)\in \R_+^d$ belongs to the set  $
    \Gamma$  if and only if for any  $ i=1, \ldots , d$,   there exists  $\theta _i=
    (\theta _i^1,\ldots,\theta_i^d) \in {\cal M}_1$ satisfying 
\be\label{2-7p} 
\overrightarrow{\gamma _i} ~<~ \sum_{j=1}^d \theta_i^j \overrightarrow{\gamma _j}, 
\ee
where the vectors $\overrightarrow{\gamma_i} = (\gamma_i^k,1 \le k \le d)$  for $i=1,
\ldots ,d$, are defined by \eqref{2-1}. 

\medskip
2) Moreover, if $\gamma_i > 0$ for all
$i=1,\ldots,d$, then $\gamma=(\gamma_1,\ldots,\gamma_d)\in\Gamma$ if for any  $ i=1, \ldots , d$,   there
    exists  $\theta _i= (\theta _i^1,\ldots,\theta_i^d) \in {\cal M}_1$ satisfying the
    following condition
\be\label{2-7} 
\gamma _i^k ~<~ \sum_{j=1}^d \theta_i^j \gamma _j^k,  \quad \text{ whenever } \;  k \in \{1, \cdots ,d\} \setminus \{i\}  \text{ and } \; Q_{ki}>0
\ee
\end{prop}

From the above proposition it follows that the set $\Gamma$ is
open in $\R_+^d$. 

The proof of this proposition is given in Section~\ref{h-banach}.

\bigskip

Our following result proves that the set $\Gamma$ is nonempty and provides an explicit form
for some of the vectors ${\bf \gamma}\in\Gamma$.  Recall that the spectral radius ${\cal
  R}$ of the routing matrix $P$ is  defined  by~:  
\[
{\cal R} ~\stackrel{def}{=}~ \inf_{\rho > 0} ~\max_{i=1,\ldots,d} ~{(\rho P)_i}/{\rho_i}
\]
where the infimum is taken over all $\rho = (\rho_1,\ldots,\rho_d)$ with positive
components $\rho_1 >0, \ldots,\rho_d >0$. If the matrix $P$ is irreducible, ${\cal R}$  is
the Perron-Frobenius 
eigenvalue and the last infimum is achieved for the left hand side Perron-Frobenius eigenvector
$\rho$ of $P$ (see Seneta~\cite{Seneta}). Under the hypothesis (A), the spectral
radius ${\cal R}$ is strictly less than unity and consequently, the set of vectors $\rho = (\rho_1,\ldots,\rho_d)$ satisfying the
inequalities 
\[ 0 ~\leq~ (\rho P)_i ~<~ \rho_i \quad \forall i\in\{1,\ldots,d\} \] 
is nonempty. Remark  that these inequalities are equivalent to
\be\label{2-5}
 (\rho P)_i ~<~ \rho_i \quad \forall i\in\{1,\ldots,d\} 
\ee
Indeed, a vector $\rho=(\rho_1,\ldots,\rho_d)$ satisfies the
inequalities \eqref{2-5}  if and only if the vector $\beta
= (\beta_1,\ldots,\beta_d) = \rho - \rho P$ has  positive components $\beta_i > 0$
for all $i=1,\ldots,d$. Now since the equality $\beta = \rho - \rho P$ is equivalent to  $\rho
~=~ \beta G$,  then \eqref{2-5} implies that  $ 0 < \rho _i$ and  $0 ~\leq~ (\rho P)_i$ for all $i=1,\ldots,d$. 

For a vector $\rho = (\rho_1,\ldots,\rho_d)$ satisfying the inequalities
\eqref{2-5},  we define 
\[
{\cal R}(\rho) ~\stackrel{def}{=}~ ~\max_{i=1,\ldots,d} ~{(\rho P)_i}/{\rho_i}
\]
and we let 
\be\label{2-6}
x_\rho ~\stackrel{def}{=}~ \sup\left\{ x > 0 :~ \log(1+x) \geq {\cal R}(\rho)~x\right\}. 
\ee

\begin{theorem}\label{pr2}  Suppose that conditions (A) and (B) are satisfied and  let a
  vector $\rho=(\rho_1,\ldots,\rho_d)$ satisfy \eqref{2-5}. Then for $\varepsilon > 0$, the
  vector 
  ${\bf \gamma}=(\gamma_1,\ldots,\gamma_d)$ defined by   
\be\label{2-8}
\gamma_i = 
\eps G_{ii}/\rho_i, \quad \text{ for all } \quad i=1,\ldots,d,
\ee
belongs to the set $\Gamma$ whenever 
\[
0 ~<~ \varepsilon ~<~  \min_{ 1\leq i\leq d} ~\frac{\rho_i}{G_{ii}} ~x_{\rho}. 
\]
\end{theorem}

\medskip

Theorem~\ref{pr1} and Theorem~\ref{pr2} provide a class of explicit Lyapunov
functions for Jackson networks. Indeed, for $\gamma_i =  \eps G_{ii}/\rho_i$, 
\[
\frac{\gamma _i}{G_{ii}} ~\Bigl( \frac{\mu_i}{1 +\gamma_i} -\nu_i \Bigr) ~=~ \eps
\left(\frac{\mu_i}{\rho_i + \eps G_{ii}} - \frac{\nu_i}{\rho_i}\right) ~>~ 0 
\]
if and only if 
\[
0 ~<~ \eps ~<~ ~\frac{\rho_i}{G_{ii}}  \left(\frac{\mu_i}{\nu_i} - 1\right).
\]
Hence,  using Theorem~\ref{pr1} together with Theorem~\ref{pr2} and the equality $G_{ji} = Q_{ji} G_{ii}$, one gets

\begin{cor}\label{cor1} Suppose that the conditions (A) and (B) are satisfied and let a
  vector 
   $\rho=(\rho_1,\ldots,\rho_d)$ satisfy the inequalities \eqref{2-5}. Then for $\varepsilon > 0$, the function 
\be\label{2-11}
h_{\eps,\rho}(x) ~=~ h_{\eps,\rho}(x^1,\ldots,x^d) ~=~ \sum_{i=1}^d ~\prod_{j=1}^d (1 + \eps G_{ji}/\rho_i)^{x^j}
\ee
satisfies 
\be\label{2-12}
\limsup_{|x|\to\infty} {\cal L}h_{\eps,\rho}(x)/h_{\eps,\rho}(x) ~=~ - ~\eps ~\min_{1\leq i \leq d}
~ \left( \frac{\mu_i}{\rho_i + \eps G_{ii}} - \frac{\nu_i}{\rho_i}\right) ~<~ 0 
\ee
whenever the vector $(\eps G_{ii}/\rho_i, 1\le i \le d)$ belongs to $\Gamma$ and 
\[
 0 ~<~ \varepsilon ~<~  
 ~\min_{ 1\leq i\leq d} ~\frac{\rho_i}{G_{ii}}\left(\frac{\mu_i}{\nu_i} - 1\right),
\]
or sufficiently, whenever
\be \label{2-9}
 0 ~<~ \varepsilon ~<~  
 ~\min_{ 1\leq i\leq d} \left\{ \min \left\{\frac{\rho_i}{G_{ii}} ~x_{\rho},~\frac{\rho_i}{G_{ii}}\left(\frac{\mu_i}{\nu_i} - 1\right) \right\}\right\}.
\ee

\end{cor} 

\medskip

In the above results, one can replace the vector $\rho$ satisfying the
inequalities \eqref{2-5} by a vector $\beta G$ with  $\beta
= (\beta_1,\ldots,\beta_d)$ having  positive components $\beta_i > 0$, since as previously
mentioned,   $\beta = \rho - \rho P$ is equivalent to  $\rho = \beta G$. Moreover,
by changing if necessary $\eps$,  one can assume 
that such a vector $\beta = (\beta_1,\ldots,\beta_d)$ defines a probability measure on the
set $\{1,\ldots,d\}$. Then for any $i=1,\ldots,d$, 
\[
\rho_i/G_{ii} ~=~ \sum_{j=1}^d \beta_j G_{ji}/G_{ii} ~=~ \sum_{j=1}^d \beta_j Q_{ji} ~~\stackrel{def}{=}~
Q_{\beta i} 
\]
is the probability that a Markov chain on $\{1,\ldots,d\}$ with transition matrix $P$
and initial distribution $\beta$ ever hits the state $i$.

\bigskip

\subsection{Estimates of the essential spectral radius}
Now, we get  some explicit estimates for the essential spectral
radius $r_e^*$.  The following lower bound  is obtained by using the 
 large deviation results of the papers \cite{Ignatiouk:01,Ignatiouk:05}. 
\begin{theorem}\label{theorem_3} Under the hypotheses (A) and (B), 
\be\label{2-13}
- \min_{1\leq i\leq d} ~\frac{1}{G_{ii}} (\sqrt{\mu_i} - \sqrt{\nu_i})^2 \leq \log r_e^* 
\ee
\end{theorem}
The proof of this Theorem is given in Section~\ref{sec4}.

\bigskip

To get an upper bound for $r_e^*$ we use Theorem~\ref{pr1} and
Theorem~\ref{pr2}. Recall that  under assumptions (A) and (B), the quantity $\theta_e^* = - \log r_e^*$ is equal to the supremum
of all $\theta >0$ for which there exists a finite set $E\subset\Z_+^d$ and 
a multiplicative Lyapunov function $f : \Z_+^d \to \R_+$ satisfying the inequality  \eqref{eq1-6} and \eqref{eq1-6p}. 
The following statement is therefore a straightforward consequence of Theorem~\ref{pr1}. 
\begin{cor}\label{cor2} Under the hypotheses (A) and (B), 
\be\label{2-14}
\log r_e^* ~\leq~ - ~\sup_{\gamma\in\Gamma}~\min_{1\leq i \leq d}
~ \frac{\gamma _i}{G_{ii}} ~\Bigl( \frac{\mu_i}{1 +\gamma_i} -\nu_i \Bigr)
\ee
\end{cor}
Recall  moreover that  $\Gamma\subset \R_+^d$ and remark that the function 
\[
\gamma ~\to~ ~\min_{1\leq i \leq d}
~ \frac{\gamma _i}{G_{ii}} ~\Bigl( \frac{\mu_i}{1 +\gamma_i} -\nu_i \Bigr) 
\] 
is continuous in $\R_+^d$. Hence, in the right hand side of \eqref{2-14}, one can replace 
the supremum over the set $\Gamma$ by the supremum over the closure $\overline\Gamma$ of the set $\Gamma$ in $\R^d$.

\bigskip

Now the question arises of  a possible equality in  \eqref{2-13}. This equality holds in particular if
the upper bound given by \eqref{2-14} coincides with the lower bound in 
\eqref{2-13}.   
In this respect, remark that for every $i=1,\ldots,d$ the maximum of the function
\[
\gamma_i  ~\to~ \gamma _i ~\Bigl( \frac{\mu_i}{1 +\gamma_i} -\nu_i \Bigr)
\] 
over $\gamma_i \in\R_+$ is achieved at the point $\gamma_i^* ~=~ \sqrt{\mu_i/\nu_i} -1$ and equals $
(\sqrt{\mu_i} - \sqrt{\nu_i})^2$. Hence, if
$\gamma^*=(\gamma^*_1,\ldots,\gamma_d^*)\in\overline\Gamma$, then one  gets equality in
\eqref{2-13}. More generally, denote by $\Delta_i$ the set of all  
$\gamma \in\R_+$ satisfying the inequality  
\[
 \frac{\gamma }{G_{ii}} ~\Bigl( \frac{\mu_i}{1 +\gamma} -\nu_i \Bigr) ~\geq~ \min_{1\leq
  j\leq d} ~\frac{1}{G_{jj}} (\sqrt{\mu_j} - \sqrt{\nu_j})^2.  
\]
Under our assumptions,  $\Delta_i$  is a closed interval  such that
$\gamma_i^*\in\Delta_i\subset ]0,\mu_i/\nu_i-1[$ and clearly, 
$\Delta_i=\{\gamma_i^*\}$ for all those $i=1,\ldots,d$ for which  
\[
\min_{1\leq j\leq d} ~\frac{1}{G_{jj}} (\sqrt{\mu_j} - \sqrt{\nu_j})^2 ~=~ \frac{1}{G_{ii}} (\sqrt{\mu_i} - \sqrt{\nu_i})^2.
\]
Hence, using the estimates \eqref{2-13} and \eqref{2-14} one  will get the equality in \eqref{2-13} 
if there exists $\gamma=(\gamma_1,\ldots,\gamma_d)\in\overline\Gamma$ with $\gamma_i\in\Delta_i$ for all
$i=1,\ldots,d$. In Section~\ref{examples}, we
give several examples where these arguments 
allow to get the equality in \eqref{2-13}. Unfortunately, in the general case, the right
hand side of \eqref{2-14} is not necessarily equal 
to the left hand side of \eqref{2-13}  (see Proposition~\ref{prop3-6} in
Section~\ref{examples} below).  In the general case, using Corollary~\ref{cor1} we obtain 

\begin{cor}\label{cor3} Under the hypotheses (A) and (B), 
\[ 
\log r_e^* ~\leq~  - ~\sup_{\rho, \eps} ~\eps ~\min_{1\leq i \leq d}
~ \left( \frac{\mu_i}{\rho_i + \eps G_{ii}} - \frac{\nu_i}{\rho_i}\right) ~<~ 0
\]
where the supremum $\sup_{\eps,\rho}$ is taken over all $\eps > 0$ and $\rho=(\rho_1,\ldots,\rho_d)$  satisfying 
\eqref{2-5}  and \eqref{2-9}. 
\end{cor}

\section{Examples}\label{examples}
In this section, we give some examples for which  the above results can be applied and in
particular,  equality  in \eqref{2-13} is obtained by using Corollary~\ref{cor2}.  

\subsection{Jackson network with a branching routing matrix \bf  P}  We will say that a
matrix  $A = (a_{ji}, \;
i,j=1,\ldots,d)$  has a {\em branching structure} if for any
$i\in\{1,\ldots,d\}$ the set $\{j \in\{1,\ldots,d\} :~ a_{ji} > 0\}$, contains at 
most one element.

Recall that under our assumptions, $p_{ii}=0$
for all $i\in\{1,\ldots,d\}$. Hence, for $d=2$, any routing matrix
$P=(p_{ij},\; i,j=1,\ldots,d)$ has a branching structure. For $d > 2$, an example of
a graph  corresponding to a branching routing matrix $P$, with  vertices $\{1,\ldots,d\}$ and ordered
edges $(i\to j)$ for $i,j\in\{1,\ldots,d\}$ such that $p_{ij} > 0$, is given in
Figure~1. 

\begin{figure}[ht]
\resizebox{7.5cm}{5cm}{\includegraphics{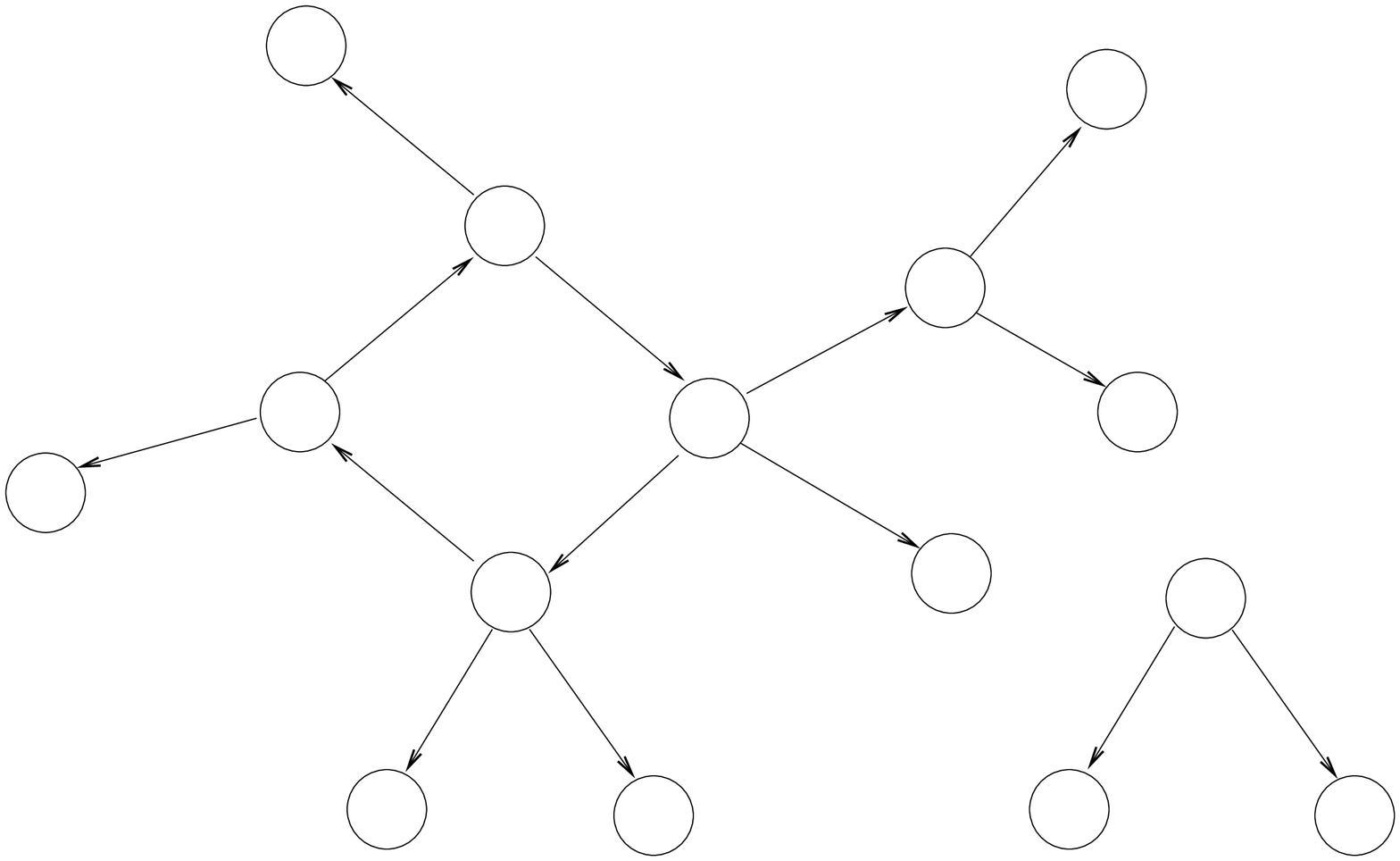}}
\caption{ }
\end{figure}  

\medskip 

\begin{prop}\label{prop3-1} Suppose that the conditions (A) and (B) are satisfied and let
  the routing matrix $P$ have a branching structure. Then the vector
  $\gamma=(\gamma_1,\ldots,\gamma_d)$ defined by \eqref{2-8} belongs to the set $\Gamma$
  for any $\eps >0$ and any vector $\rho=(\rho_1,\ldots,\rho_d)$ satisfying the inequalities \eqref{2-5}. 
\end{prop}
\begin{proof} We get this statement as a consequence of the second assertion of Proposition~\ref{Gamma}.  Indeed, let a vector $\rho$ satisfy the inequalities \eqref{2-5}. Consider a vector $\gamma = (\gamma_1,\ldots,\gamma_d)$ defined by \eqref{2-8} with some given  $\eps > 0$. Then obviously, $\gamma_i > 0$ for all $i\in\{1,\ldots,d\}$. Let us show that under the hypotheses of our proposition, \eqref{2-7} holds for any $i\in\{1,\ldots,d\}$.  If $i\in\{1,\ldots,d\}$ is such that 
$p_{ji} = 0$ for all $j\in\{1,\ldots,d\}\setminus\{i\}$ then also $Q_{ji} = 0$ for all $j\not=i$ and consequently \eqref{2-7} is trivial.  

Suppose now that for $i\in\{1,\ldots,d\}$, there is $j\in\{1, \ldots,d\}$ such that $p_{ji} > 0$. Then under the hypotheses of our proposition, such an index $j$ is unique, $j\not=i$, and 
\[ 
G_{ki} ~ =~ G_{kj}p_{ji} \quad \forall \; k\in\{1, \ldots,d\}\setminus\{i\}.
\] 
Moreover, from \eqref{2-5} it follows that $\rho_i > (\rho P)_i = \rho_j p_{ji}$ and consequently, 
\[
\gamma_i^k  ~ =~  \log(1 + \eps G_{ki}/\rho_i) ~ =~ \log(1 + \eps G_{kj} p_{ji}/\rho_i) ~ <~ \log(1 + \eps G_{kj} /\rho_j) ~ =~ \gamma_j^k
\]
for all those $k\in\{1, \ldots,d\}\setminus\{i\}$ for which $G_{kj} > 0$ or equivalently $Q_{ki} > 0$.  The last relations show that  \eqref{2-7} holds  with a unit vector $\theta_i = (\theta_i^1,\ldots,\theta_i^d)$ where $\theta_i^k = 1$ for $k\not= j$ and $\theta_i^j = 1$. Using therefore  the second assertion of Proposition~\ref{Gamma},  we conclude that $\gamma\in\Gamma$.
\end{proof}

 When combined with Theorem~\ref{pr1}, the above proposition implies the following particular version of Corollary~\ref{cor1}.

\begin{cor}\label{cor3-1} Suppose that the conditions (A) and (B) are satisfied and let the
  routing matrix have a branching structure. Suppose moreover that a vector
  $\rho=(\rho_1,\ldots,\rho_d)$ satisfies the inequalities \eqref{2-5}. Then the function $h_{\eps,\rho}$ defined by
  \eqref{2-11}  satisfies the inequality \eqref{2-12} 
  for any 
\[
0 ~<~ \eps ~<~ \min_{1\leq i \leq d}\frac{\rho_i}{G_{ii}}\left(\frac{\mu_i}{\nu_i} - 1\right).
\]
\end{cor}
From the last statement we obtain 
\begin{prop}\label{prop3-2} Suppose that the conditions (A) and (B) are satisfied and let
  either the routing matrix $P$ or its transposed matrix $^t\!P$ have a branching structure. Then 
\be\label{3-1}
\log r_e^* ~=~ - \min_{1\leq i\leq d} ~\frac{1}{G_{ii}} \left(\sqrt{\mu_i} - \sqrt{\nu_i}\right)^2 
\ee
\end{prop} 

\begin{proof} 
Suppose first that the routing matrix has a branching structure. Then by
  Corollary~\ref{cor3-1}, 
\be\label{3-2}
\log r_e^* ~\leq~ - \eps ~\min_{1\leq i \leq d} \left( \frac{\mu_i}{\rho_i + \eps G_{ii}}
- \frac{\nu_i}{\rho_i}\right) 
\ee
 for any $\eps > 0$ and any vector $\rho = (\rho_1,\ldots,\rho_d)$ satisfying the
 inequalities \eqref{2-5}, 
 or equivalently (see the
remark below Corollary~\ref{cor1}) for any $\rho = \beta G$ with
$\beta=(\beta_1,\ldots,\beta_d)$ having strictly positive components $\beta_i > 0$ for all 
$i=1,\ldots,d$. By taking  the limits  as $\beta_i\to 0$ for some indices
$i\in\{1,\ldots,d\}$ one gets \eqref{3-2} also for any  vector $\rho = \beta G$  with $\beta_i\geq
0$ and $\rho_i >0$ for all $i=1,\ldots,d$. Using again the equivalence between $\rho=\beta G$ and  $\beta =
\rho - \rho P$, these arguments prove the upper bound \eqref{3-2} for any  vector
$\rho=(\rho_1,\ldots,\rho_d)$ satisfying the inequalities 
\be\label{3-3}
 0 ~<~ \rho_i  \quad \text{and} \quad (\rho P)_i ~\leq~ \rho_i, \quad \forall i=1,\ldots,d. 
\ee
Remark now that according to the definition of the traffic equations \eqref{eq1-2}, 
\[
\nu_i ~=~ (\nu P)_i + \lambda_i ~\geq~ (\nu P)_i, \quad \forall i=1,\ldots, d.
\]
If a vector $\rho=(\rho_1,\ldots,\rho_d)$  satisfies the inequalities \eqref{3-3}, then
for $\tilde\rho = (\tilde\rho_1,\ldots, \tilde\rho_d)$ with
$\tilde\rho_i=\sqrt{\nu_i\rho_i}$, by  Schwarz inequality,
\[
(\tilde\rho P)_i ~=~ \sum_j \sqrt{\nu_j\rho_j} ~p_{ji} ~\leq~
\sqrt{\sum_j \nu_j
  p_{ji}}\sqrt{\sum_j \rho_j p_{ji}} ~\leq~ \sqrt{\nu_i\rho_i} = \tilde\rho_i, \quad
\forall i=1,\ldots,d,
\]
and consequently one can replace the quantities $\rho_i$ at the right hand side of
\eqref{3-2} by $\tilde\rho_i=\sqrt{\nu_i\rho_i}~$ (recall that $\nu _i>0$ for all $i=1, \ldots, d$, hence $\tilde\rho_i>0$ for all $i=1, \ldots, d$). The resulting inequality 
\[
\log r_e^* ~\leq~ - \eps ~\min_{1\leq i \leq d} \left( \frac{\mu_i}{\sqrt{\nu_i\rho_i} + \eps G_{ii}}
- \sqrt{\frac{\nu_i}{\rho_i}}\right) 
\]
with 
\[
\eps ~=~ \min_{1\leq i\leq d} ~\frac{\sqrt{\rho_i}}{G_{ii}}
~\left(\sqrt{\mu_i} - \sqrt{\nu_i}\right) 
\]
provides the following upper bound  
\be\label{3-4}
\log r_e^* ~\leq~ - ~\min_{1\leq i\leq d} ~\frac{\sqrt{\rho_i}}{G_{ii}}
~\left(\sqrt{\mu_i} - \sqrt{\nu_i}\right) \times~\min_{1\leq i\leq d} ~\frac{1}{\sqrt{\rho_i}}
~\left(\sqrt{\mu_i} - \sqrt{\nu_i}\right). 
\ee

Moreover, if a routing matrix $P$ has a branching structure, then for any
$i\in\{1,\ldots,d\}$, either $p_{ji}=0$ for all $j\in\{1,\ldots,d\}$ and
consequently, 
\[
\sum_{j=1}^d G_{jj} p_{ji}  = 0 ~\leq~ G_{ii},
\]
or else there is a unique $j\in\{1,\ldots,d\}$  such that $p_{ji} > 0$ and consequently,
\[
\sum_{k=1}^d G_{kk} p_{ki} ~=~ G_{jj}p_{ji} = G_{ji} \leq G_{ii}. 
\]
These relations show that the vector $\rho=(\rho_1,\ldots,\rho_d)$ with $\rho_i = G_{ii}$ satisfies
the inequalities \eqref{3-3}. Using \eqref{3-4}  with this vector $\rho$ we obtain 
\[
\log r_e^* ~\leq~ - ~\min_{1\leq i\leq d} ~\frac{1}{G_{ii}}
~\left(\sqrt{\mu_i} - \sqrt{\nu_i}\right)^2.
\]
The last inequality combined with \eqref{2-13} proves \eqref{3-1}.

Suppose now that the transposed matrix $^t\!P$ has a branching structure, and let us show
that in this case, the equality \eqref{3-1} also holds. For this we apply a time
reversing argument to the Markov process $(Z(t))$. The time reversed
Markov process $(\tilde{Z}(t))$ is generated by  
\[
 \tilde{\cal L}f
(y) = \sum_{z\in \Z^N_+} \tilde{q}(y,z) (f(z) - f(y)), \quad y\in \Z^N_+, 
\]
with 
\[
\tilde{q}(y,z) ~=~ \pi(z)q(z,y)/\pi(y).
\]
A straightforward calculation shows that this is also a Jackson network but with different
parameters: the arrivals at the $i$-th queue 
are Poisson with parameter $\tilde\lambda_i = \nu_ip_{i0}$, the 
services delivered by the server are exponentially distributed with the same parameter
$\tilde\mu_i = \mu_i$ as for the original Jackson network $(Z(t))$, and the routing matrix 
$(\tilde{p}_{ij}, i,j=0,\ldots,d)$ is given by  $\tilde{p}_{i0} ~=~ \lambda_i/\nu_i$ and $
\tilde{p}_{ij} ~=~ \nu_j p_{ji}/\nu_i$ for $i,j\in\{1,\ldots,d\}$. Under our assumptions, the time reversed Markov process
$(\tilde{Z}(t))$ also satisfies the conditions (A) and (B) with the same solution
$(\nu_i,\, i=1,\ldots,d)$ of the
traffic equations  and the same stationary probabilities $(\pi(x);\,x\in
\Z^N_+)$. Moreover, for any finite subset $E\subset\Z_+^d$, letting 
\[
\tilde\tau_E ~=~ \inf\{ t > 0 : \tilde{Z}(t) \in E\} \quad \text{ and } \quad \tau_E ~=~ \inf\{ t > 0 : Z(t) \in E\} 
\]
one gets 
\[
 \P_x( \tilde{Z}(t) = y, \; \tilde\tau_E > t ) ~=~\pi(y) \P_y( Z(t) = x, \; \tau_E
> t )/\pi(x), \quad \forall x,y\in\Z_+^d\setminus E
\]
and consequently, the essential spectral radius of the time reversed Markov process
$(\tilde{Z}(t))$ is the same as for the original Markov process $(Z(t))$. If the
transposed matrix $^t\!P$ has a branching structure, then the routing matrix $\tilde{P}
=(\tilde{p}_{ij}, \; i,j=1,\ldots,d)$ has the same property and consequently, the
above arguments applied to the time reversed Markov process $(\tilde{Z}(t))$ prove the
equality \eqref{3-1}. 
\end{proof}

\medskip

\subsection{Jackson networks with a completely symmetrical routing matrix $P$} 
Now we consider a Jackson
network having a completely symmetrical  routing matrix $P=(p_{ij}, \, i,j=1,\ldots,d)$
with $p_{ij} = p < 1/(d-1)$ for all $i\not=j$, $i,j\in\{1,\ldots,d\}$. Then $Q_{ij} = p/(1-(d-2)p)  \stackrel{def}{=} q$ for all $i\not=j$, $i,j\in\{1,\ldots,d\}$, where  $0<q<1$. The following
proposition provides an explicit form for the set $\Gamma$ in this case. To formulate this
result, it is convenient to introduce the  function 
\[
\Sigma(\gamma_1,\ldots,\gamma_d) ~=~ \sum_{j=1}^d \frac{\max_{1\leq i\leq d}
  ~\log(1 + q\gamma_i) - \log(1+q\gamma_j)}{\log(1+\gamma_j) - \log(1+q\gamma_j)}  
\]
for $\gamma \in \R_+^d$ satisfying $\gamma_j>0$ for all $j=1,\ldots,d$ (note that for such a $\gamma$, since $q<1$,   then $\log(1+\gamma_j) > \log(1+q\gamma_j)$  for all $j=1,\ldots,d$ and the above quantity is well-defined).

\begin{prop}\label{prop3-3} Suppose the conditions (A) and (B) are satisfied and let $p_{ij} = p < 1/(d-1)$ for all $i\not=j$,
  $i,j\in\{1,\ldots,d\}$.  Then  $\gamma=(\gamma_1,\ldots,\gamma_d)\in\Gamma$ if and only if $\gamma_i > 0$ for all
  $i\in\{1,\ldots,d\}$ and   
\be\label{3-5}
\Sigma(\gamma_1,\ldots,\gamma_d) ~<~ 1. 
\ee
\end{prop}
\begin{proof} For any $\gamma \in \R_+^d$, any $i\in\{1,\ldots,d\}$ and a non-zero vector
  $v=(v^1,\ldots,v^d)\in\R_+^d$ with $v^i=0$, letting $| v|= \sum _j v_j >0$, the inequality \eqref{2-2} becomes 
\[
|v|\log(1+q\gamma_i)~<~ \max_{1\leq j\leq d} \Bigl( v^j \log(1+\gamma_j) +
(|v| - v^j)\log(1+q\gamma_j)  \Bigr)
\]
or equivalently,  
\be\label{3-6}
|v|\log\frac{1+q\gamma_i}{1+q\gamma_j} ~<~ v^j \log\frac{1+\gamma_j}{1+q\gamma_j} \; \text{ for some } j\in\{1,\ldots,d\}
\ee
Since $q < 1$,  the inequality  \eqref{3-6}  is  trivially satisfied when 
\[\gamma_i < \max_j \gamma_j
\]
Thus, $\gamma\in\Gamma$ if and only if  \eqref{3-6} holds for any $i\in \{1,\ldots,d\}$ such that
\be\label{3-7}
 \gamma_i = \max_j \gamma_j
\ee
and for any non-zero vector $v\in\R_+^d$
with $v^i=0$. 

\medskip
Consider now a vector $\gamma =(\gamma_1,\ldots,\gamma_d)$
with $\gamma_i > 0$ for all $i\in\{1,\ldots,d\}$.  If $\gamma\not\in\Gamma$, then  using the above arguments it follows that for some index $i\in \{1,\ldots,d\}$ satisfying the equality \eqref{3-7} there is a non-zero vector $v\in\R_+^d$
with $v^i=0$ such that 
\[
|v|\log\frac{1+q\gamma_i}{1+q\gamma_k} ~\geq~ v^k \log\frac{1+\gamma_k}{1+q\gamma_k} \; \text{ for all} \;  k\in\{1,\ldots,d\}, 
\]
and consequently, 
\[
\frac{\max_j\log(1+q\gamma_j) - \log(1+q\gamma_k)} { \log(1+\gamma_k) -
\log(1+q\gamma_k)}\, |v| ~\geq~  v^k \quad \text{for all $k\in\{1,\ldots,d\}$}. 
\]
Summing  these inequalities proves that for such a vector $\gamma$,  \eqref{3-5}
fails to hold. 

\medskip
Conversely, suppose that $\gamma\in\Gamma$.  Then  \eqref{3-6} holds for any index $i\in \{1,\ldots,d\}$ satisfying the equality \eqref{3-7} and for any non-zero vector $v\in\R_+^d$
with $v^i=0$. From \eqref{3-6} it follows that $\gamma$ is non-zero. Moreover,   let  $i\in\{1,\ldots,d\}$ satisfy \eqref{3-7}. Then for any $k\in\{1,\ldots,d\}\setminus\{i\}$, 
using the inequality \eqref{3-6} with a unit vector $v = (v^1,\ldots,v^d)$ such that $v^k = 1$ and $v^j = 0$ for $j\not=k$, one gets 
\[
q \max_j \gamma_j  ~=~ q \gamma_i ~<~  \gamma_k,
\]
and consequently, $\gamma _k >0$ for all $k \in \{1, \ldots, d\}$.  The quantity $\Sigma(\gamma_1,\ldots,\gamma_d)$ is therefore well-defined and equal to $|v|$ for  $v=(v^1,\ldots,v^d)\in\R_+^d$ given by 
 \[
v^j ~=~ \frac{\max_{1\leq i\leq d}
  ~\log(1+q\gamma_i) - \log(1+q\gamma_j)}{\log(1+\gamma_j) -
  \log(1+q\gamma_j)}, \quad \text{ \; $j\in\{1,\ldots,d\}$}. 
\]
If $|v| =\Sigma(\gamma_1,\ldots,\gamma_d) =0$, then \eqref{3-5} obviously holds. Otherwise, using again \eqref{3-6} with such a vector $v$ and with any $i$ satisfying \eqref{3-7} gives
\[ \Sigma(\gamma_1,\ldots,\gamma_d) ~\max_{1\leq i\leq d}
  ~\log\frac{1+q\gamma_i}{1+q\gamma_j}    ~< ~\max_{1\leq i\leq d}
  ~\log\frac{1+q\gamma_i}{1+q\gamma_j} \quad \text{ for some } j\in\{1,\ldots,d\}
  \]
which proves   \eqref{3-5}.

\end{proof}

Remark  that for a completely symmetrical routing matrix $P$, 
\[
G_{ii} = \sum_{n=0}^\infty \left(p\sum_{j\not=i} Q_{ji}\right)^n ~=~ \left(1 - \frac{(d-1)p^2}{1-(d-2)p}\right)^{-1}
\]
Hence, when combined with Theorem~\ref{pr1}, the above proposition implies the following statement, similar to   Corollary 2.1.
\begin{cor}\label{cor3-2} Under the hypotheses of Proposition~\ref{prop3-3}, the function 
\[
h_\gamma(x) = \sum_{i=1}^d \exp(\overrightarrow{\gamma_i}\cdot x) 
~=~ \sum_{i=1}^d
(1+\gamma_i)^{ x^i}  (1 + q\gamma_i )^{\sum_{j\not= i}x^j}
\]
satisfies  
\[
\limsup_{|x|\to\infty} \frac{{\cal L}h_{\gamma}(x)}{h_{\gamma}(x)} = -  \left(1 - \frac{(d-1)p^2}{1-(d-2)p}\right) \min_{1\leq i \leq d}
~\gamma _i~\Bigl( \frac{\mu_i}{1 +\gamma_i} -\nu_i \Bigr)<~ 0 
\]
whenever \eqref{3-5} holds and $0 < \gamma_i <\displaystyle{ \frac{\mu_i}{\nu_i}-1}$ for all
$i=1,\ldots,d$. 
\end{cor}

Note that  \eqref{3-5} is satisfied for any vector $\gamma \in \R^d$ such that $\gamma_1= \ldots  =\gamma_d >0$, so that the set of vectors $\gamma \in \R^d$ satisfying both   \eqref{3-5} and  $0 < \gamma_i < \displaystyle{ \frac{\mu_i}{\nu_i}-1}$ for all 
$i=1,\ldots,d$ is nonempty. Using therefore Theorem~\ref{theorem_3} and Corollary~\ref{cor2} we obtain 
\begin{cor}\label{cor3-3} Under the hypotheses of Proposition~\ref{prop3-3},  
\begin{multline*}
-\left(1 - \frac{(d-1)p^2}{1-(d-2)p}\right) \min_{1\leq i \leq d} \left(\sqrt{\mu_i}
-\sqrt{\nu_i}\right)^2 ~\leq~ \log r^*_e \\~\leq~ - \left(1 -
\frac{(d-1)p^2}{1-(d-2)p}\right) \sup_\gamma~\min_{1\leq i \leq d}~ \gamma _i~\Bigl( \frac{\mu_i}{1 +\gamma_i} -\nu_i \Bigr) ~<~ 0 
\end{multline*}
where the supremum is taken over all $\gamma \in \Gamma$, or equivalently, over all $\gamma =
(\gamma_1,\ldots,\gamma_d)$ with $\gamma_i > 0$ for all 
$i=1,\ldots,d$ such that  inequality \eqref{3-5} holds. 
\end{cor}
Thus,  if the conditions of Proposition~\ref{prop3-3} are satisfied and 
\be\label{3-8}
\sup_{\gamma\in\Gamma}~\min_{1\leq i \leq d}~ \gamma _i~\Bigl( \frac{\mu_i}{1 +\gamma_i} -\nu_i \Bigr) ~=~
\min_{1\leq i \leq d}  (\sqrt{\mu_i}-\sqrt{\nu_i})^2,
\ee
then  
\be\label{3-9}
\log r^*_e  = - \left(1 - \frac{(d-1)p^2}{1-(d-2)p}\right) \min_{1\leq i \leq d} \left(\sqrt{\mu_i}
-\sqrt{\nu_i}\right)^2,
\ee
that is, relation \eqref{3-1} again holds. The following statement gives some simple sufficient conditions for the equalities \eqref{3-8} and \eqref{3-9} 

\begin{cor}\label{cor3-4} Suppose that  for some $i_0 \in \{1, \ldots, d \}$,
\be\label{3-10}
\min_{1\leq i\leq d} \left(\sqrt{\mu_i}-\sqrt{\nu_i}\right) ~=~ \sqrt{\mu_{i_0}}-\sqrt{\nu_{i_0}}.
\ee
and
\be\label{3-11}
\min_{1\leq i\leq d}~ \left( \frac{\mu_i}{\sqrt{\mu_{i_0}}} - \frac{\nu_i}{\sqrt{\nu_{i_0}}}\right)  ~=~ \sqrt{\mu_{i_0}} -
\sqrt{\nu_{i_0}}. 
\ee
Then under the hypotheses of Proposition~\ref{prop3-3},  \eqref{3-9} holds. 

 \noindent  In particular, \eqref{3-9} holds  if one of the following 
conditions  is satisfied :
\begin{enumerate}
\item[(i)] $\mu_i/\nu_i = \mu_j/\nu_j$ \; for all \; $i,j\in\{1,\ldots,d\}$, 
\item[(ii)]  there is $i_0\in\{1,\ldots,d\}$ such that $\mu_i \geq \mu_{i_0}$ and
  $\nu_i\leq\nu_{i_0}$ for all $i\in\{1,\ldots,d\}$. 
\end{enumerate} 

\end{cor}

\begin{proof} Here, as noted above, any vector  $\gamma=(\gamma_1,\ldots,\gamma_d)$ with $\gamma_1 = \ldots = \gamma_d >
0$ belongs to the set $\Gamma$. Hence, by Corollary~\ref{cor3-3}, the equality \eqref{3-9} holds if  
\be\label{3-12}
\sup_{t > 0}~\min_{1\leq i \leq d}~ t~\Bigl( \frac{\mu_i}{1 +t} -\nu_i \Bigr)~=~
\min_{1\leq i \leq d}  \left(\sqrt{\mu_i}-\sqrt{\nu_i}\right)^2. 
\ee
Recall that the maximum of the function $t\in \R_+\to \displaystyle{t~\Bigl( \frac{\mu_i}{1 +t} -\nu_i \Bigr)}$ is
achieved at the point $\gamma^*_i=\sqrt{\mu_i/\nu_i}-1$ and equals
$\left(\sqrt{\mu_i}-\sqrt{\nu_i}\right)^2$. Hence,  assuming   \eqref{3-10}, then \eqref{3-12} holds if and only if 
\[
\gamma ^*_{i_0}~\Bigl( \frac{\mu_i}{1 +\gamma^*_{i_0}} -\nu_i \Bigr) ~\geq~
\left(\sqrt{\mu_{i_0}}-\sqrt{\nu_{i_0}}\right)^2, \quad \forall i\in\{1,\ldots,d\}.
\]
Since 
\[
\gamma ^*_{i_0}~\Bigl( \frac{\mu_i}{1 +\gamma^*_{i_0}} -\nu_i \Bigr)  ~=~ \left(\frac{\mu_i}{\sqrt{\mu_{i_0}}} -
\frac{\nu_i}{\sqrt{\nu_{i_0}}}
\right) \left(\sqrt{\mu_{i_0}}-\sqrt{\nu_{i_0}}\right),
\]
the last inequalities are equivalent to \eqref{3-11}. 
 
 Now if condition  (i) is satisfied, consider $i_0$ such that $\min_{1\leq i\leq d} \nu_i ~=~ \nu_{i_0} $, then  \eqref{3-10}  is satisfied. Using $\mu_i= \mu_{i_0} \nu_i/\nu_{i_0}$, we get 
\[
\min_{1\leq i\leq d}~ \left( \frac{\mu_i}{\sqrt{\mu_{i_0}}} - \frac{\nu_i}{\sqrt{\nu_{i_0}}}\right) ~=~ \min_{1\leq i\leq d}~  \frac{\nu _i}{\nu_{i_0}}( \sqrt{\mu_{i_0}} -
\sqrt{\nu_{i_0}}) ~=~\sqrt{\mu_{i_0}} -
\sqrt{\nu_{i_0}}. 
\]
 so that  \eqref{3-11} holds, hence also  \eqref{3-9}   from  the first part of the proof.
 
Finally, if  condition (ii) is satisfied, then  $i_0$  clearly satisfies   \eqref{3-10}  and \eqref{3-11},  so that  
 \eqref{3-9} again  follows from the first part of the corollary. 
 \end{proof}

\medskip

Remark that (ii)  is in particular satisfied if $\mu_i = \mu_j$ for all  $i,j\in\{1,\ldots,d\}$, 
 or  if $\nu_i = \nu_j$  for all $i,j\in\{1,\ldots,d\}$.

\bigskip

Our following result is a necessary and sufficient condition for the equality \eqref{3-8}. 
Denote 
\[
m ~=~ \min_{0 \le i \le d} (\sqrt{\mu_i}-\sqrt{\nu_i})^2 
\]
and consider  for $i\in\{1,\ldots,d\}$, 
\[
\Delta_i = \left\{ t\in\R_+ : ~\displaystyle{t~\Bigl( \frac{\mu_i}{1 +t} -\nu_i \Bigr)}\geq  m\right\}.
\]
A straightforward calculation shows that $\Delta_i = [a_i,b_i]$ with 
\[
a_i ~=~ \frac{\mu_i - \nu_i - m - \sqrt{(\mu_i + \nu_i -
     m)^2 - 4 \nu_i\mu_i} }{2\nu_i}
\]
and
\[
b_i ~=~ \frac{\mu_i - \nu_i - m+ \sqrt{(\mu_i + \nu_i -
      m)^2 - 4 \nu_i\mu_i} }{2\nu_i}.
\]
Moreover, 
\[
\sqrt{\mu_i/\nu_i} -1\in \Delta_i\subset\{t\in\R_+ : ~ \displaystyle{t~\Bigl( \frac{\mu_i}{1 +t} -\nu_i \Bigr)}> 0\} = \left]0,\displaystyle{\frac{\mu_i}{\nu_i}-1}\right[,
\] 
and consequently, 
\[
b_i ~\geq~ \sqrt{\mu_i/\nu_i}-1 ~\geq~ a_i > 0,
\]
where $b_i=a_i = \sqrt{{\mu_i}/{\nu_i}}-1$ if and only if $(\sqrt{\mu_i}-\sqrt{\nu_i} )^2=
m$. 
We put 
\[
\widehat{a} ~=~ \max_{1\leq i \leq d} a_i \quad \text{ 
and } \quad \widehat\gamma_i = \min\{b_i, \widehat{a}\} 
\]
for $i=1,\ldots,d$. 
\medskip

\begin{prop}\label{prop3-4} Suppose that the conditions of Proposition~\ref{prop3-3} are
  satisfied. Then \eqref{3-8} holds if and only if
  $\Sigma(\widehat\gamma_1,\ldots,\widehat\gamma_d) \leq 1$.
\end{prop}
\begin{proof} Indeed, suppose first that \eqref{3-8} holds and remark that  for any $i\in\{1,\ldots,d\}$, since  $\nu_i >0$, the function $
  \displaystyle{t~\Bigl( \frac{\mu_i}{1 +t} -\nu_i \Bigr)} \to - \infty$ as $t\to +\infty$. These functions being continuous on $\R_+$,  it follows that the function 
\[
\min_{1\leq i \leq d}    \gamma _i~\Bigl( \frac{\mu_i}{1 +\gamma_i} -\nu_i \Bigr)  
\]
attains its maximum over the closure $\overline{\Gamma}$ of the set $\Gamma$ at some point
$\widetilde\gamma\in\overline\Gamma$. Moreover, relation \eqref{3-8}   proves that 
\[
\min_{1\leq i \leq d}   \widetilde  \gamma _i~\Bigl( \frac{\mu_i}{1 +\widetilde\gamma_i} -\nu_i \Bigr)  
 ~=~
 \min_{0 \le i \le d} (\sqrt{\mu_i}-\sqrt{\nu_i})^2 \stackrel{def}{=} m,
\]
from which it follows  that $\widetilde\gamma_i\in\Delta_i$  and
consequently, $\widetilde\gamma_i > 0$ for all $i\in\{1,\ldots,d\}$. The quantity
$\Sigma(\widetilde\gamma_1,\ldots,\widetilde\gamma_d)$ is therefore well defined and by
Proposition~\ref{prop3-3}, 
\[
\Sigma(\widetilde\gamma_1,\ldots,\widetilde\gamma_d) \leq 1.
\]
To prove that $\Sigma(\widehat\gamma_1,\ldots,\widehat\gamma_d)\leq 1$ it is now sufficient to
show that $(\widehat\gamma_1,\ldots,\widehat\gamma_d)$ achieves the minimum of
the function $\Sigma(\gamma_1,\ldots,\gamma_d)$ over $(\gamma_1,\ldots,\gamma_d)\in
\Delta_1\times\cdots\times\Delta_d$. For this let us notice that this function is 
  continuous on the compact set $\Delta_1\times\cdots\times\Delta_d$ and hence attains 
  its minimum on this set at some point $\gamma^*=(\gamma_1^*,\ldots,\gamma_d^*)$. 
  
  If $\Sigma(\gamma_1^*,\ldots,\gamma_d^*) = 0$, then from the definition of the function $\Sigma$ it follows that  $\gamma^*_j = \max_j \gamma_i^*$ for all $j=1,\ldots, d$,  that is, the intervals $\Delta _i$, $i=1, \cdots d$, have some common point $t = \gamma_1^* = \cdots = \gamma_d^*$. But in this case, 
  \[
  \min_i b_i ~\geq~  t ~\geq~  \max_{1\leq i\leq d}  a_i ~\stackrel{def}=~ \widehat{a}
  \]
  from which, using the definition of the vector $\widehat\gamma$, it follows that $\widehat \gamma _1 = \ldots = \widehat \gamma _d = \widehat a$ and consequently, also $\Sigma(\widehat \gamma_1,\ldots, \widehat \gamma_d) = 0$. 
  
 Suppose now that $\Sigma(\gamma_1^*,\ldots,\gamma_d^*) > 0$ and let us show that in this case, $\gamma^* = \widehat\gamma$. Indeed, in this case, from the definition of the function $\Sigma(\gamma_1,\ldots,\gamma_d)$ it follows that  $\gamma^*_j < \max_j \gamma_i^*$ for some $j=1,\ldots, d$. Moreover, 
  \be\label{e3-?}
\max_{1\leq i\leq d} \gamma^*_i = \max_{1\leq i\leq d}  a_i ~\stackrel{def}=~\widehat{a}. 
\ee
because otherwise, one could find some $\epsilon>0$ for which the vector $\gamma '=(\gamma '_1,\ldots,\gamma '_d)$ given by
  \[
\gamma '_j ~=~ \begin{cases} 
\gamma^*_j - \epsilon &\text{ if $\gamma^*_j = \max_{1\leq i\leq d}\gamma^*_i$,}  \\
\gamma^*_j &\text{ if $\gamma^*_j < \max_{1\leq i\leq d}\gamma^*_i$}   \\
\end{cases}
\]
belongs to the set  $\Delta_1\times\cdots\times\Delta_d$  and satisfies $\Sigma(\gamma '_1,\ldots,\gamma '_d) <
\Sigma(\gamma_1,\ldots,\gamma_d)$.  Remark now that   the following two assertions are equivalent : 
\begin{itemize}
\item[(i)] $(\gamma_1,\ldots,\gamma_d)\in \Delta_1\times \cdots \times \Delta_d$ and $\max_j \gamma_j = \widehat{a}$
\item[(ii)] $a_j \leq \gamma_j \leq \min\{b_i, \widehat a\} ~\stackrel{def}=~ \widehat{\gamma_j}$ for all $j\in\{1,\ldots,d\}$.
\end{itemize} 
Moreover, for any point $\gamma = (\gamma_1,\ldots,\gamma_d)$ satisfying the inequalities (ii), 
\[
\Sigma(\gamma_1,\ldots,\gamma_d) = \sum_{j=1}^d \frac{\log(1+q\hat{a}) -
  \log(1+q\gamma_j)}{\log(1+ \gamma_j) - \log(1+q\gamma_j)}.
  \]
 When combined with \eqref{e3-?}, these remarks show  that the point $\gamma^* = (\gamma^*_1,\ldots,\gamma^*_d)$ achieves  the minimum of the function 
\[
\sum_{j =1}^d \frac{\log(1+q\hat{a}) -
  \log(1+q\gamma_j)}{\log(1+ \gamma_j) - \log(1+q\gamma_j)}
\] 
over the set  $[a_1,\widehat{\gamma_1}] \times \cdots \times [a_d,\widehat{\gamma_d}]$.    The function 
  \[
 t ~\to~  \frac{\log(1+q\hat{a}) -
  \log(1+qt)}{\log(1+ t) - \log(1+qt)}
  \]
  being decreasing on $]0, \widehat{a}]$, from this it follows that $
 \gamma^*_j = \widehat\gamma_j$ for all  $j\in\{1,\ldots,d\}$ and consequently, 
\[
\Sigma(\widehat\gamma_1,\ldots,\widehat\gamma_d) ~=~ \Sigma(\gamma_1^*,\ldots,\gamma_d^*) ~\leq~
\Sigma(\widetilde\gamma_1,\ldots,\widetilde\gamma_d) ~\leq~ 1. 
\]

\bigskip 
Conversely, suppose that $\Sigma(\widehat\gamma_1,\ldots,\widehat\gamma_d)\leq 1$ and let us prove
the equality \eqref{3-8}. We know from Section 2.2 that 
\begin{align*}
\min_{1\leq i \leq d}  (\sqrt{\mu_i}-\sqrt{\nu_i})^2 &~=~ \sup_{\gamma\in\R_+^d}~\min_{1\leq i \leq d}~   \gamma _i~\Bigl( \frac{\mu_i}{1 +\gamma_i} -\nu_i \Bigr)   \\ &~\geq~ \sup_{\gamma\in\overline\Gamma}~\min_{1\leq i \leq d}~   \gamma _i~\Bigl( \frac{\mu_i}{1 +\gamma_i} -\nu_i \Bigr)  . 
\end{align*}
Moreover,  since $\widehat{\gamma}_i \in \Delta_i$ for all $i \in \{1, \ldots, d\}$,  
\[
\min_{1\leq i \leq d}~   \widehat \gamma _i~\Bigl( \frac{\mu_i}{1 +\widehat \gamma_i} -\nu_i \Bigr)    ~=~ 
\min_{1\leq i \leq d}  (\sqrt{\mu_i}-\sqrt{\nu_i})^2
\]
To get \eqref{3-8} it is therefore sufficient to show that  $\widehat\gamma\in\overline\Gamma$. 
If $\Sigma(\widehat\gamma_1,\ldots,\widehat\gamma_d) < 1$, then $\widehat\gamma\in\Gamma$  by
Proposition~\ref{prop3-3}. 
Suppose now that $\Sigma(\widehat\gamma_1,\ldots,\widehat\gamma_d) = 1$. Then
clearly
\[
\min_{1\leq i\leq d} \widehat\gamma_i < \max_{1\leq i\leq d} \widehat\gamma_i,
\]
and letting 
\[
\gamma_i{(\eps)} = \begin{cases} \widehat\gamma_i - \eps &\text{ if $\widehat\gamma_i =
    \max_{1\leq i\leq d} \widehat\gamma_i$,}\\
\widehat\gamma_i &\text{ otherwise,}
\end{cases} 
\]
one gets $\gamma_i{(\eps) }> 0$ for all $i\in\{1,\ldots,d\}$ and
  $\Sigma(\gamma_1{(\eps)},\ldots,\gamma_d{(\eps)}) < 1$ for all $\eps > 0$ small
enough. By Proposition~\ref{prop3-3},  it follows that
$\left(\gamma_1{(\eps)},\ldots,\gamma_d{(\eps)}\right)\in\Gamma$ for all $\eps > 0$ small enough
and consequently, letting $\eps\to 0$ we conclude that $\widehat\gamma\in\overline\Gamma$. 
\end{proof}

\bigskip
 
The last result of this section provides an example where \eqref{3-8} fails to
hold. This example shows that unfortunately,  in general, the left hand side of
\eqref{2-13} and the right hand side of \eqref{2-14} are not necessarily equal. 

\begin{prop}\label{prop3-5} Suppose that the following conditions are satisfied : 
\begin{itemize}
\item[(i)] $d >3$; 
\item[(ii)]  $\lambda_1+\ldots +\lambda_d = 1$ and $0 = \lambda_1 < \lambda_i$ for all
  $i\in\{2,\ldots,d\}$;
\item[(iii)] $\sqrt{\mu_i}-\sqrt{\nu_i} = t > 0$ for all $i\in\{1,\ldots,d\}$.
\end{itemize}
Then under the hypotheses of
  Proposition~\ref{prop3-3}, for any $p>0$ small enough, there is $t_p > 0$ such that for $t
  > t_p$, the inequality \eqref{3-8} fails to hold.
\end{prop}
\begin{proof} By Proposition~\ref{prop3-4}, it is sufficient to show
that  
\be\label{3-13}
\lim_{p\to 0} ~\lim_{t\to\infty} ~\Sigma(\widehat\gamma_1,\ldots,\widehat\gamma_d) ~>~ 1. 
\ee
Remark that under the hypotheses of Proposition~\ref{prop3-5}, $a_i = b_i = 
 \sqrt{\mu_i/\nu_i}-1$ and consequently, 
\[
\widehat\gamma_i ~=~ \sqrt{\mu_i/\nu_i}-1 ~=~  \frac{t}{\sqrt{\nu_i}} ,
\]
for all $i=1,\ldots,d$. 
Moreover, a straightforward calculation shows that for any $i=1,\ldots,d$, 
\[
\nu_i ~=~ \frac{1}{p+1} \left( \lambda_i + \frac{p}{1+ p - d p} \sum_{i=1}^d\lambda_i
\right) ~=~ \frac{1}{p+1} \left( \lambda_i + \frac{p}{1+ p - d p} \right).
\]
Since under the hypotheses of our proposition, $\lambda_1 = 0 < \lambda_i$ for all $i\in\{2,\ldots,d\}$, the above relations show that  $\max_i \widehat\gamma_i ~=~ \widehat\gamma_1$.  Using the definition of $\Sigma(\gamma_1,\ldots,\gamma_d)$ we conclude therefore that 
\begin{align*}
\lim_{t\to\infty} ~\Sigma(\widehat\gamma_1,\ldots,\widehat\gamma_d)  = \lim_{t\to\infty}
\sum_{i=2}^d \frac{\log\left((1 + q t/\sqrt{\nu_1})/(1 + q
  t/\sqrt{\nu_i})\right)}{\log\left((1 + 
  t/\sqrt{\nu_i})/(1 + q t/\sqrt{\nu_i})\right)} ~=~ \sum_{i=2}^d
\frac{\log\left(\nu_i/\nu_1\right)}{2 \log\left(1/q \right)} 
\end{align*}
where $q ~=~ p/(1 + 2 p - d p)$ and 
\[
\nu_i/\nu_1 ~=~ 1 + (1 + p - d p) \lambda_i/ p, \quad \forall i=2,\ldots,d. 
\]
Hence, 
\[\lim_{t\to\infty} ~\Sigma(\widehat\gamma_1,\ldots,\widehat\gamma_d)  = \frac {\log \prod_{i=2}^d \left( \frac{\lambda_i}{p} +1-(d-1) \lambda _i \right)}{2 \log \left(\frac{1}{p} +2-d\right)}
\]  
and since $\lambda _i >0$ for $i=2, \cdots ,d$, 
\[
\lim_{p\to 0} ~\lim_{t\to\infty} ~\Sigma(\widehat\gamma_1,\ldots,\widehat\gamma_d) ~=~ \frac{d-1}{2}
\]
Under the hypothesis (i), the last relation proves  \eqref{3-13}. 
\end{proof}

\subsection{Jackson network with three nodes on a circle} Consider a Jackson network with
three nodes ($d=3$) and a routing matrix 
\be\label{3-15} 
P = \left(\begin{matrix}
0& p& q\\
q& 0& p\\
p& q& 0
\end{matrix}\right) \; \text{ with $0 < p < q < 1$ such that
$p+q < 1$.}
\ee
\begin{figure}[ht]
\resizebox{7.5cm}{5cm}{\includegraphics{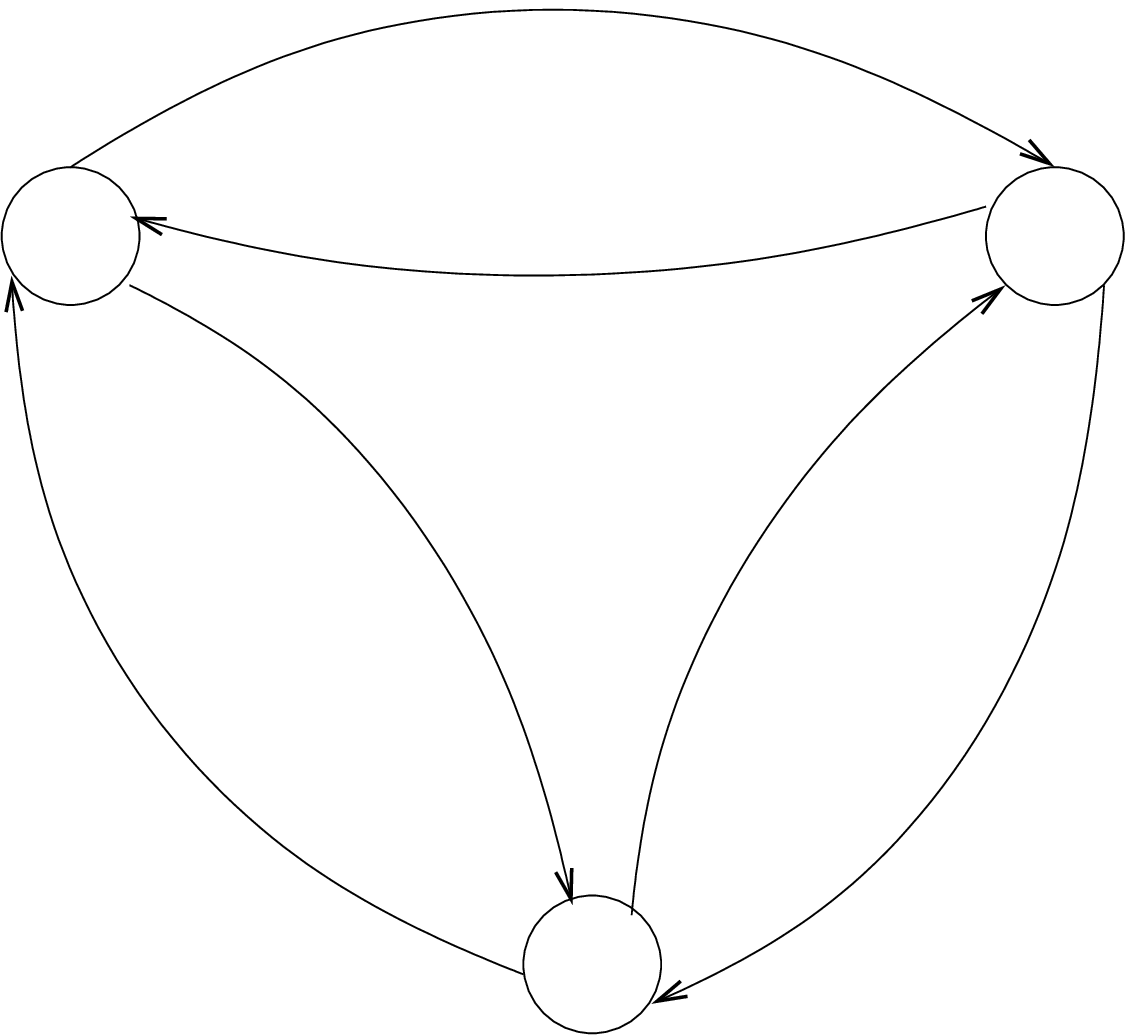}}
\put(-122,115){$q$}
\put(-122,145){$p$}
\put(-85,64){$q$}
\put(-142,65){$q$}
\put(-181,49){$p$}
\put(-35,49){$p$}
\put(-103,8){$3$}
\put(-18,108){$2$}
\put(-202,108){$1$} 
\caption{ }
\end{figure} (see Figure~2). 
Here, as a consequence of Corollary~\ref{cor2} and
Theorem~\ref{theorem_3} we get 
\begin{prop}\label{prop3-6} Suppose that a Jackson network with
three nodes and a routing matrix \eqref{3-15} satisfies conditions (A) and (B). Then 
\begin{multline}\label{3-16}
- ~\frac{1- p^3 - q^3 - 3pq}{1-
  pq}  \min_i \left(\sqrt\mu_i - \sqrt\nu_i\right)^2  ~\leq~ \\ \log r_e^* ~\leq~ - ~\frac{1- p^3 - q^3 - 3pq}{1-
  pq}  ~\sup_{t > 0}~\min_{1\leq i \leq d}~ t
\Bigl( \frac{\mu_i}{1 +t} -\nu_i \Bigr).
\end{multline}
If moreover the equalities \eqref{3-10} and \eqref{3-11} hold for some $i_0 \in \{1, \ldots, d\}$, then 
\be\label{3-17}
\log r_e^* ~=~ - ~\frac{1- p^3 - q^3 - 3pq}{1-
  pq}  \min_i \left(\sqrt\mu_i - \sqrt\nu_i\right)^2.  
\ee
In particular  \eqref{3-17} holds if at least one of the conditions (i) or
  (ii) of Corollary~\ref{cor3-4} is satisfied. 
  \end{prop}
\begin{proof}
Indeed, a straightforward calculation shows that 
\[
G ~\stackrel{def}{=}~ (Id - P)^{-1} ~=~ \frac{1}{1-p^3-q^3-3pq} 
\left(\begin{matrix}
1-pq& q^2+p& p^2+q\\
p^2+q& 1-pq& q^2+p\\
q^2+p& p^2+q& 1-pq
\end{matrix}
\right).
\]
The first inequality of \eqref{3-16} is therefore a straightforward consequence of
Theorem~\ref{theorem_3}. By Corollary~\ref{cor2},  to prove the second inequality  of \eqref{3-16} it is 
sufficient to show that for any $t>0$, the vector ${\bf
  \gamma}=(t,t,t)$  belongs to
the set $\Gamma$. For this let us first notice that
under the hypotheses of our proposition, the matrix of hitting probabilities $Q=(Q_{ij}, \; i,j=1,2,3)$ is given by 
\[
Q ~=~ \frac{1}{1-pq} \left(\begin{matrix} 
1-pq& q^2+p& p^2+q\\
p^2+q& 1-pq& q^2+p\\
q^2+p& p^2+q& 1-pq
\end{matrix}\right)
\]
Without any restriction of generality we can
assume that $p\leq q$. Then 
\[
Q_{31} = Q_{12} = Q_{23} ~\leq~ Q_{21} = Q_{32} = Q_{13} < 1 
\]
and consequently, for $\gamma =(t,t,t)$ with $t>0$ 
and any $v=(v_1,v_2,v_3)\in\R_+^3$ with $v_1=0$ and $(v_1,v_2) \neq (0,0)$, one gets 
\begin{align*}
\overrightarrow{\gamma_1}\cdot v &~=~v_2  \log\bigl(1  + Q_{21}t\bigr)+ v_3 \log\bigl(1 +
Q_{31}t \bigr) \\ 
&~<~ \begin{cases}
 v_2 \log (1+t)+  v_3\log\bigl(1 +
Q_{32}t \bigr) ~=~ \overrightarrow{\gamma_2}\cdot v &\text{ if $v_2 > 0$,} \\
v_3 \log(1+t)~=~ \overrightarrow{\gamma_3} \cdot v &\text{ if $v_3 > v_2  = 0$.}
\end{cases}
\end{align*}
from which it follows that 
\[
\overrightarrow{\gamma_1}\cdot v  ~<~ \max_{j} \overrightarrow{\gamma_j}\cdot v. 
\]
Permuting indices shows  that for any $i\in\{1,2,3\}$ and any non-zero vector  $v=(v_1,v_2,v_3)\in\R_+^3$ with
$v_i =0$,  
\[
\overrightarrow{\gamma_i}\cdot v ~<~ \max_{j} \overrightarrow{\gamma_j}\cdot v, \quad \text{ if }  \gamma =(t,t,t)  \text{ with } t>0
\] 
Hence, for any $t > 0$,  the vector ${\bf \gamma} = (t,t,t)$  belongs to the set $\Gamma$ and consequently,  by Corollary~\ref{cor2}, the
second inequality of \eqref{3-16} is also verified. The first part of our proposition is therefore proved.  The
second part of Proposition~\ref{prop3-6} follows from \eqref{3-16} by using  the same arguments as in the proof of
Corollary~\ref{cor3-4}. 
\end{proof}

\section{Background}
For a given $\Lambda\subset\{1,\ldots,d\}$, denote  $\Lambda^c = \{1,\ldots,d\}\setminus\Lambda$ and consider the sets 
$\R_+^{\Lambda,d} ~\stackrel{def}{=}~ \{ x\in\R^d : x^j\geq 0,   \; \, \forall j\not\in\Lambda\}$  and 
\[
{\cal B}_\Lambda ~\stackrel{def}{=}~  \left\{ \alpha = (\alpha^1, \ldots,\alpha^d) \in\R^d ~:~ \;  \alpha^i\leq \log\Bigl( \, \sum_{j=1}^d p_{ij}e^{\alpha^j} +
p_{i0} \Bigr),  \; \, \forall i\not\in\Lambda  \right\}.
\] 
For $\beta \in\R^{\Lambda,d}_{+} $ and $i,j\in \{ 0,1,\ldots ,d\}, \; i\not= 0,$ we define 
\[
m^\Lambda_{ij}(\beta) ~\stackrel{def}{=}~ p_{ij}e^{-\beta^i} + \sum_{n\geq 1} \sum_{j_1,
\ldots ,j_n \in
\Lambda^c} p_{ij_1}p_{j_1j_2}\cdots p_{j_nj}\exp \left(- \beta^i - \sum_{k=1}^n
\beta^{j_k}\right).
\] 
The following result  provides a suitable homeomorphism from the set $\R_+^{\Lambda,d} $ onto ${\cal B}_\Lambda$, this is a straightforward  consequence of Proposition~8.1 of the paper~ \cite{Ignatiouk:01}. 
\begin{prop}\label{prop_background}({\em  Proposition~8.1~ \cite{Ignatiouk:01}}) Under hypothesis (A), 
\begin{itemize}
\item[--] for any
$\Lambda\subset\{1,\ldots,d\}$ and $\beta\in\R_+^{\Lambda,d} $,  the system of equations 
\[
\begin{cases}
\beta^i = \alpha^i, \quad \text{ for } \; i\in\Lambda,\\
\beta^i =  \log\left( \sum_{j=1}^d p_{ij} e^{\alpha^j - \alpha^i} +
p_{i0}e^{-\alpha^i}\right), \; \text{ for } \; i\in\{1,\ldots,d\}\setminus\Lambda
\end{cases}
\]
has a unique solution $\alpha=\alpha_\Lambda(\beta)\in{\cal B}_\Lambda$ :
\begin{equation}\label{solution} \left\{ \begin{array}{lll}
\alpha^i (\beta) & = & \beta^i \; \mbox{ for } i \in \Lambda,\\
\alpha^i(\beta) & = & \log
\left(\sum_{j\in \Lambda} m^\Lambda_{ij}(\beta) e^{\beta^j} +
m^\Lambda_{i0}(\beta)\right) \; \mbox{
for } i \in \Lambda^c, \end{array} \right. \end{equation} 
\item[--] the mapping $\beta\to\alpha_\Lambda(\beta)$ determines a homeomorphism
from $\R_+^{\Lambda,d}$ onto the set ${\cal B}_\Lambda$;
\item[--] the function $R(\alpha_\Lambda(\beta))$
is strictly convex in $\R_+^{\Lambda,d}$. 
\end{itemize}
\end{prop} 
This result will be used to investigate  the different Laplace transforms  of the jump distribution on the different ``faces"  of the space $\Z_+^d$.
For $\Lambda \subset  \{1, \cdots, d\}$ and $\alpha \in \R^d$, the Laplace transform of the jump distribution corresponding to the face $\Lambda$ is defined by  
\[
R_{\Lambda} (\alpha) ~\stackrel{def}{=}~ \sum _{j=1}^d \lambda _j (e^{\alpha
         ^j} -1) + \sum _{j \in \Lambda} \mu _j \left( \sum _{k=1}^d
       p_{jk}e^{ \alpha ^k - \alpha^j}+ p_{j0}e^{  - \alpha^j} -1\right)   
\]
and for $\Lambda=\{1,\ldots,d\}$, we denote  
\[
R(\alpha)  ~\stackrel{def}{=}~ R_{\{1,\ldots,d\}} (\alpha) ~=~ \sum _{j=1}^d \lambda _j (e^{\alpha
         ^j} -1) + \sum _{j = 1}^d \mu _j \left( \sum _{k=1}^d
       p_{jk}e^{ \alpha ^k - \alpha^j}+ p_{j0}e^{  - \alpha^j} -1\right). 
\]
As a consequence of the above proposition one gets the following statement . 
\begin{lemma}\label{lemma_background}
Under hypothesis (A), for any $i \in \{1, \cdots, d\}$ and $s \in ~]-1, + \infty[$ the system of equations 
\be\label{system}
e^{\alpha^i}= 1 + s, \quad e^{\alpha^j} ~=~ \sum_{k=1}^d p_{jk} e^{\alpha^k} + p_{j0}, \quad j\in\{1,\ldots,d\}\setminus\{i\}
\ee
has a unique solution $\alpha  ~=~ \alpha(s)= (\alpha^1(s), \cdots ,\alpha^d(s))$ given by 
\[
\alpha^j(s)  ~=~  \log\left( 1+Q_{ji} s \right) , \quad j\in\{1,\ldots,d\}\setminus\{i\}.
\]
Moreover, for any $\Lambda \subset  \{1, \cdots, d\}$, this solution  satisfies the equality 
\[
 R_{\Lambda}(\alpha(s))  ~=~ \frac{s}{G_{ii}} \left(
\nu_i - \frac{\mu_i}{1 + s} \1 _{\Lambda}(i) \right).
\]  
\end{lemma}
\begin{proof}
Indeed, for $\beta =(\beta^1,\ldots,\beta^d)$ with  $\beta^i = \log(1 + s)$ and $\beta^j = 0$ for $j\not= i$, one gets 
\[
m^{\{i\}}_{ji}(\beta) = Q_{ji}  \quad \text{ and } \quad m^{\{i\}}_{j0}(\beta)  = 1 - Q_{ji} 
\]
for all $j\in\{1,\ldots,d\}$.  Hence, the first assertion of Lemma~\ref{lemma_background} is a straightforward consequence of Proposition~\ref{prop_background}.  Moreover,  for any $\Lambda\subset\{1,\ldots,d\}$, 
\begin{align*}
R_\Lambda(\alpha(s))  &=  \sum _{j=1}^d \lambda _j \left(e^{\alpha^j(s)} -1\right) +  \mu _i \left( \sum _{k=1}^d
       p_{ik}e^{ \alpha^k(s) - \alpha^i(s)}+ p_{j0}e^{  - \alpha^i(s)} -1\right)  \1 _{\Lambda}(i)  \\
       &=  \sum _{j=1}^d \lambda _j Q_{ji} s +  \frac{\mu _i}{1 + s}\left( \sum _{k=1}^d
       p_{ik}e^{ \alpha^k(s) }+ p_{j0} - 1 - s\right)  \1 _{\Lambda}(i) \\
      &=  s \sum _{j=1}^d \lambda _j Q_{ji}  +  \frac{ \mu _i s}{1 + s}\left( \sum _{k=1}^d
       p_{ik} Q_{ki}  - 1\right)  \1 _{\Lambda}(i) 
\end{align*}
The last equality combined with the relations 
\[
\sum _{j=1}^d \lambda _j Q_{ji}    ~=~  \frac{1}{G_{ii}} \sum _{j=1}^d \lambda _j G_{ji}   ~=~ \frac{\nu_i}{G_{ii}} \quad \text{and}  \quad  1 - \sum _{k=1}^d
        p_{ik} Q_{ki}  ~=~ \frac{1}{G_{ii}} 
\]
proves the second assertion of Lemma~\ref{lemma_background} 
\end{proof}

\section{Proof of Theorem~\ref{pr1}}\label{sec3}
We begin the proof of this theorem with the following lemma.

\begin{lemma}\label{lem4-1} For any  $\gamma\in\R_+^d$ and $1\leq i\leq d$, the function
  $f_i(x) = \exp(\overrightarrow{\gamma_i}\cdot x)$ satisfies the equality 
\be\label{4-1}
{\cal L}f_i (x) ~=~  \frac{\gamma_i}{G_{ii}}\left( 
\nu_i - \1_{\{x_i > 0\}} \frac{\mu_i}{1+\gamma_i} \right)  f_i(x), \quad x\in\Z_+^d
\ee
\end{lemma}
\begin{proof} A straightforward calculation shows that for any  $\alpha= (\alpha
  ^1,\ldots,\alpha^d) \in\R_+ ^d$, the exponential function $f_\alpha(x) = \exp(\alpha
  \cdot x ) $ satisfies the equality 
\[
 {\cal L}f_\alpha (x) ~=~ R_{\Lambda(x)} (\alpha)  f_\alpha (x),
\] 
where for $x\in\R_+^d$, we denote by  $\Lambda(x)$  the set of all $j\in\{1,\ldots,d\}$ for
which $x^j > 0$ and  for $\Lambda\subset\{1,\ldots, d\}$, 
\[
R_{\Lambda} (\alpha) ~=~ \sum _{j=1}^d \lambda _j (e^{\alpha
         ^j} -1) + \sum _{j \in \Lambda} \mu _j \left( \sum _{k=1}^d
       p_{jk}e^{ \alpha ^k - \alpha^j} + p_{j0}e^{  - \alpha^j}  -1\right).   
\]
Furthermore,  by Lemma~\ref{lemma_background}, from the definition of the vector $\overrightarrow{\gamma_i}$ it follows that  
$\alpha = (\alpha^1,\ldots,\alpha^d) = \overrightarrow{\gamma_i}$ is the unique solution of the system \eqref{system} for $s=\gamma_i$ and 
\[
R_{\Lambda(x)} (\overrightarrow{\gamma_i}) ~=~ \frac{\gamma_i}{G_{ii}}\left( 
\nu_i - \1_{\{x_i > 0\}} \frac{\mu_i}{1+\gamma_i} \right) 
\]
The equality \eqref{4-1} is therefore verified. 
\end{proof}
Now we are ready to complete the proof of Theorem~\ref{pr1}. For the function $h_\gamma$
defined by \eqref{2-3}, Lemma~\ref{lem4-1} proves that 
\begin{align}
{\cal L}h_\gamma(x) 
&=~  \sum_{i\in\Lambda(x)} \frac{\gamma_i}{G_{ii}}\left( 
\nu_i -  \frac{\mu_i}{1+\gamma_i} \right) \exp( \overrightarrow{\gamma_i }\cdot x) +
\sum_{i\not\in\Lambda(x)} \frac{\nu_i\gamma_i}{G_{ii}} \exp(
\overrightarrow{\gamma_i }\cdot x) \label{4-2} \\
&\leq~  \max_{i\in\Lambda(x)} ~ \frac{\gamma_i}{G_{ii}} \left(\nu_i  - \frac{\mu_i}{1+\gamma_i}\right) ~h_\gamma(x)  +
\sum_{i\not\in\Lambda(x)} \frac{\nu_i}{G_{ii}}\gamma_i \exp(
\overrightarrow{\gamma_i} \cdot x).\nonumber
\end{align} 
To get the inequality 
\be\label{4-3}
\limsup_{|x|\to\infty} {\cal L}h_\gamma (x)/h_\gamma(x) ~\leq~ \max_{1\leq i \leq d}
~\frac{\gamma_i}{G_{ii}} \left(\nu_i  - \frac{\mu_i}{1+\gamma_i}\right) 
\ee
it is therefore sufficient to show that 
\[
\frac {\max _{i \not\in \Lambda (x)} \exp(
\overrightarrow{\gamma_i }\cdot x)}{ \max _{i \in \Lambda (x)} \exp(
\overrightarrow{\gamma_i }\cdot x)} ~\to~ 0 \quad \text{ as } \quad
|x|\to\infty, 
\]
or equivalently that 
\be\label{4-4} 
\lim _{|x| \to \infty} 
\exp \left (  \max _{i\not\in \Lambda(x) } \overrightarrow{\gamma_i} \cdot x - \max _{i \in \Lambda(x) }
\overrightarrow{\gamma_i} \cdot x  \right ) =0.
\ee
The last relation follows from the definition of the set $\Gamma$. Indeed, using the
inequality \eqref{2-2} with $v=x/|x|$ for an arbitrary $x\in\R_+^d\setminus\{0\}$, one gets 
\[
\overrightarrow{\gamma_i} \cdot \frac{x}{|x|} ~<~ \max_{j\not= i} \overrightarrow{\gamma_j} \cdot \frac{x}{|x|}, \quad  \quad
\forall  i\not\in\Lambda(x)
\]
from which it follows that 
\[
\max _{i \not\in \Lambda(x) } \overrightarrow{\gamma_i} \cdot \frac{x}{|x|} < \max_{j=1,\ldots,d}  
\overrightarrow{\gamma_j }\cdot \frac{x}{|x|} ~=~ \max_{j\in\Lambda(x)} \overrightarrow{\gamma_j} \cdot \frac{x}{|x|} 
\]
for any non-zero $x\in\R_+^d$.  The function 
\[
x\to \max _{i \not\in \Lambda(x) } \overrightarrow{\gamma_i }
\cdot x - \max_{i\in\Lambda(x)} \overrightarrow{\gamma_i}  \cdot x ~=~ \max _{i \not\in \Lambda(x) } \overrightarrow{\gamma_i} 
\cdot x  - \max_{j=1,\ldots,d}  \overrightarrow{\gamma_j }\cdot x
\]
being upper semi-continuous on the compact set
$S_+^d = \{x\in\R_+^d : |x| = 1\}$, from the above inequality it follows  that 
\be\label{4-5}
\max _{i \not\in \Lambda(x) } \overrightarrow{\gamma_i }
\cdot x - \max_{i \in\Lambda(x)} \overrightarrow{\gamma_i}  \cdot x ~<~ - \delta |x|, \quad \forall x\in \R_+^d\setminus\{0\}
\ee
with some $\delta > 0$, and consequently, \eqref{4-4} holds. The inequality \eqref{4-3} is therefore proved. 
Moreover, \eqref{4-2} and \eqref{4-5} applied for $x\in\R_+^d$
with $\Lambda(x)=\{i\}$ prove that 
\[
\limsup_{|x|\to\infty, \; \Lambda(x) =\{i\}} {\cal L}h_\gamma (x)/h_\gamma(x)  ~=~
\frac{\gamma_i}{G_{ii}} \left(\nu_i  - \frac{\mu_i}{1+\gamma_i}\right). 
\]
Using this relation together with \eqref{4-3} one gets \eqref{2-4}.

\section{Proof of Proposition~\ref{Gamma} } \label{h-banach}
We begin the proof of Proposition~\ref{Gamma} with the following lemma. 
\begin{lemma}\label{Hahn-B}
For $u_1, \cdots ,u_d\in\R^d$, the following two  properties are equivalent:
\begin{enumerate}
\item for any $v\in \R^d_+ \setminus\{0 \}$, there exists some $i \in \{1, \cdots ,d\}$ such that $u_i \cdot v >0$,
\item  there exists  some $\theta  = (\theta^1, \cdots , \theta ^d) \in {\cal M}_1$ such
  that $\sum_{j=1}^d \theta ^j u_j > 0$.
\end{enumerate}
\end{lemma}
\begin{proof} It is straightforward that $(2) \Rightarrow (1)$,  since for $\theta$ satisfying condition
$(2)$ and  for any  $v\in \R^d_+ \setminus\{0 \}$,  
\[ 0 < v \cdot \sum_{j=1}^d \theta ^j u_j  = \sum_{j=1}^d \theta ^j\, u_j  \cdot v \]
so that one of the non-negative terms of the last sum needs to be positive. 

To prove the converse, assume that $(2)$ is not satisfied, so that  for any $\theta \in
{\cal M}_1$, 
\[
\sum_{j=1}^d \theta^j u_j \not\in ~]0, + \infty[^d
\]
This means that the two convex subsets of $\R^d$ given by the open orthant  $]0, + \infty[^d$ on one hand, and
    the closed convex  cone  $C$ generated by  vectors $u_1, \cdots , u_d$ on the other
    hand, that is, $C= \left\{  \sum_{j=1}^d \theta_j u_j~,~ \theta=(\theta_1, \cdots ,\theta_d) \in \R_+^d\right\} $,
    are disjoint. Then by Hahn-Banach theorem,  there exists some hyperplane
    separating these  two convex sets, that is, there exists some $v  \in \R^d\setminus
    \{0\}$ and some $c\in \R$ such that 
\[ ]0, + \infty[^d \subset  \{x\in \R^d: x\cdot v \ge c\} \quad \text{and} \quad C  \subset   \{x\in \R^d: x\cdot v \le c\}.\]
Note that the first inclusion extends to the closed orthant   $[0, + \infty[^d$. Now since
    the zero vector  is  both in $C$ and in the closed orthant,  the constant $c$ must be
    zero. The first inclusion, extended to $[0, + \infty[^d$ and applied to the canonical
        vectors  $e_i$ for $ i=1, \cdots ,d$, yields that $v$ has non-negative
        components. And the second inclusion above implies in particular that $u_i\cdot v
        \le 0$ for all $i$, proving that $(1)$ is not satisfied. 
\end{proof}

\medskip
We are now  ready to complete the proof of Proposition~\ref{Gamma}.

For $\gamma \in \R_+^d$, the condition $\gamma \in \Gamma$, described by the inequalities
~\eqref{2-2}, says that for any $i \in \{1, \cdots ,d\}$, the property $(1)$ of  the lemma
is satisfied, with $d-1$ in place of  $d$ and with,   as vectors $u_i$'s, the $d-1$
projections on $\R ^{\{1, \cdots ,d\} \setminus \{i\}}$ of the  vectors
$\overrightarrow{\gamma_j}-\overrightarrow{\gamma_i},~ j \neq i$. The lemma  thus proves
that $\gamma \in \Gamma$ is equivalent to existence for each $i= 1, \cdots , d$, of some
$\theta _i \in {\cal M}_1$ satisfying $\theta _i^i=0$ and  for all $k \in \{1, \ldots ,d\}
\setminus \{i\}$, 
\be\label{eq5-1}
\gamma _i^k ~<~ \sum_{j=1}^d \theta_i^j \gamma _j^k. 
\ee
 It is straightforward that the condition $\theta _i^i=0$ can be removed. The first part
 of Proposition~\ref{Gamma} is therefore proved. 
 
 Suppose now that $\gamma > 0$, and that
 for any  $ i=1, \ldots ,d$,   there exists some $\theta _i= (\theta
 _i^1,\ldots,\theta_i^d)\in {\cal M}_1$, satisfying  the  inequalities \eqref{eq5-1} for all
 those indices $k$ for which  $Q_{ki}>0$. The
 inequalities  \eqref{eq5-1} being strict, without any restriction of 
 generality, one can assume that $\theta_i^j>0$ for all
 $i,j\in\{1,\ldots,d\}$. Then for  all $i, k \in\{1,\ldots,d\}$ for which $Q_{ki}=0$,  
 \[ \gamma _i^k  ~=~ 0 ~<~ \theta_i^k \log(1 + \gamma_k)  ~=~ \theta_i^k\gamma^k_k ~\leq~ \sum_{j=1}^d \theta_i^j \gamma _j^k.
\]
The inequalities \eqref{eq5-1} hold therefore  for all $i,k\in\{1,\ldots,d\}$,  which ensures that $\gamma \in \Gamma$.

\section{Proof of Theorem~\ref{pr2}}\label{sec4}
Suppose that the conditions of Theorem~\ref{pr2} are satisfied. For  the vector ${\bf
  \gamma} = (\gamma_1,\ldots,\gamma_d)$ defined by \eqref{2-8},  
it follows from ~\eqref{2-1} that
\[\gamma^j_i ~=~ \log(1 + \eps G_{ji}/\rho_i) ~\geq~ 0, \quad \forall i,j\in\{1,\ldots,d\}.
\]
Since $\gamma _i >0$ for all $i=1, \cdots, d$, then from the second assertion of 
Proposition~\ref{Gamma}, $\gamma \in \Gamma$ if for every $i$, ~\eqref{2-7} is
satisfied  with some vector $\theta _i \in {\cal M}_1$.  
Let $i \in \{1, \cdots ,d\}$ and  $k\not= i$ be such that $Q_{ki}>0$. Then,  there is $j\in\{1,\ldots, d\}\setminus\{i\}$ such that $p_{ji} > 0$ and consequently, $(\rho P)_i > 0$. Letting $\theta_i^j = \rho_j p_{ji}/(\rho P)_i$ for $j=1,\ldots,d$  we obtain
\begin{align}
\gamma_i^k &~=~ \log(1 + \eps G_{ki}/\rho_i) ~=~ \log\left(1 + \eps \sum_{j=1}^d
\frac{G_{kj} p_{ji}}{ \rho_i} \right) \le   \eps \sum_{j=1}^d  \frac{G_{kj} p_{ji}}{\rho_i} \nonumber\\
&~=~  \frac{(\rho P)_i}{\rho_i}   \sum_{j=1}^d \eps \frac{G_{kj} p_{ji}}{(\rho
  P)_i} ~=~ \frac{(\rho P)_i}{\rho_i}  \sum_{j=1}^d   \theta_i^j\, \eps \,
\frac{G_{kj}}{\rho_j}  \nonumber\\
&~\le~ {\cal R}(\rho) \sum_{j=1}^d   \theta_i^j \, \eps
\,\frac{G_{kj}}{\rho_j} . \label{6-1}
\end{align}
Assuming now that 
\[
0 ~<~ \varepsilon ~<~  \min_{ 1\leq i\leq d} ~\frac{\rho_i}{G_{ii}} ~x_{\rho} ,
\]
one gets 
 \[ 0 ~\le~ \eps \frac{G_{kj}}{\rho_j} ~<~x_{\rho}   \quad \text{ for all } j \in \{1, \cdots ,d\},\]
 where the left inequality is strict at least for some $j \in \{1, \cdots , d\}$ with
 $p_{ji}>0$, because $Q_{ki}>0$ implies that   $G_{ki} ~=~ \sum_{j=1}^d 
G_{kj} p_{ji} >0$. It then results from the definition of $x_{\rho}$ that
 \[
  {\cal R}(\rho)~ \eps~\frac{G_{kj}}{\rho_j} ~\le \log\left(1+ \eps \frac{G_{kj}}{\rho_j}\right) \quad \text{ for all } j \in \{1, \cdots ,d\},
 \]
 where the inequality is strict at least for some $j$ with $\theta_i^j>0$.
 The last inequality combined with \eqref{6-1} proves that 
\[ 
\gamma_i^k ~<~ \sum_{j=1}^d  \theta_i^j \log\left(1+ \eps \frac{G_{kj}}{\rho_j}\right)~=   \sum_{j=1}^d  \theta_i^j
\gamma_j^k.
\]
The condition ~\eqref{2-7} of the Proposition~\ref{Gamma} is thus satisfied  and therefore, $\gamma \in \Gamma$.

\section{Proof of Theorem~\ref{theorem_3}}\label{sec5} 
To prove Theorem~\ref{theorem_3}, we use the equality \eqref{eq1-9} and the explicit representation of the sample path large
deviation rate function $I_{[0,T]}(\phi)$ obtained in
\cite{Ignatiouk:01,Ignatiouk:04}. Recall that the family of scaled processes $Z_\eps(t) =
\eps Z(t/\eps), \; t\in[0,T]$ satisfies the sample path large deviation principle (see
~\cite{A-D,D-E, Ignatiouk:01, Ignatiouk:04}) with the
good rate function 
\[
I_{[0,T]}(\phi) = \begin{cases}
\int_0^T L(\phi(t),\dot{\phi}(t)) \, dt &\text{if $\phi :[0,T]\to\R_+^d$ is absolutely
  continuous}\\
+\infty &\text{otherwise} 
\end{cases}
\]
where the local rate function $L(x,v)$ is given by the formula (see~\cite{Ignatiouk:01}) 
\[
L(x,v) ~\stackrel{def}{=}~ \sup_{\alpha\in{\cal B}_{\Lambda(x)}} \bigl(\alpha \cdot v -
R(\alpha) \bigr), \quad \forall v\in\R^d, \; x\in\Z_+^d.
\]
As above, $\alpha\cdot v$ denotes here the usual scalar product of $\alpha$ and $v$ in $\R^d$,
\[
R(\alpha) ~\stackrel{def}{=}~ \sum_{i=1}^d \mu_i\Bigl( \sum_{j=1}^d p_{ij}e^{\alpha^j-\alpha^i} +
p_{i0}e^{-\alpha^i} - 1\Bigr) + \sum_{i=1}^d \lambda_i(e^{\alpha^i}-1),
\]
for $x=(x^1,\ldots,x^d)\in\R_+^d$, 
\[
\Lambda(x) ~\stackrel{def}{=}~ \{i\in\{1,\ldots,d\} : x^i > 0\}
\]
and ${\cal B}_\Lambda$ is the set of all those  $\alpha = (\alpha^1, \ldots,\alpha^d) \in\R^d$
for which  
\[
e^{\alpha^i} \leq  \, \sum_{j=1}^d p_{ij}e^{\alpha^j} +
p_{i0}  \quad \text{ for all $ i\not\in\Lambda$.}
\]
For a constant function $\phi_x(t) \equiv x$ with $x\in (x^1,\ldots,x^d)\in\R^d_+$, we get 
\[
I_{[0,1]}(\phi_x) ~=~ - \inf_{\alpha\in{\cal B}_{\Lambda(x)}} R(\alpha) 
\]
and  using \eqref{eq1-9} we obtain 
\begin{align*}
\log r_e^* &\geq~ - \inf_{x\in\R_+^d :~ x\not= 0} ~I_{[0,1]}(\phi_x) ~=~ 
\max_{\Lambda\subset\{1,\ldots,d\}, \; \Lambda\not=\emptyset} ~\inf_{\alpha\in{\cal B}_{\Lambda}} R(\alpha)\\ 
&\geq~ \max_{1\leq i\leq d}
~\inf_{\alpha\in{\cal B}_{\{i\}}} R(\alpha).
\end{align*}
To prove Theorem~\ref{theorem_3} it is therefore sufficient to show that 
\be\label{7-2}
\max_{1\leq i\leq d}~\inf_{\alpha\in{\cal B}_{\{i\}}} R(\alpha) ~=~ - \min_{1\leq i\leq d}
~\frac{1}{G_{ii}}(\sqrt{\mu_i}-\sqrt{\nu_i})^2.
\ee
For this we first notice that for any $i\in\{1,\ldots,d\}$, 
\begin{align} 
\inf_{\alpha\in{\cal B}_{\{i\}}} R(\alpha)  &=~ \inf\left\{ R(\alpha) \; \Bigl| \; \alpha \in \R^d ,   \ e^{\alpha^j} \leq  \, \sum_{k=1}^d p_{jk}e^{\alpha^k} +
p_{j0} , \quad \forall \, j\not= i \right\}\nonumber \\
&\leq~ \inf\left\{ R(\alpha) \; \Bigl| \; \alpha \in \R^d ,   \ e^{\alpha^j}  =  \, \sum_{k=1}^d p_{jk}e^{\alpha^k} +
p_{j0} , \quad \forall \, j\not= i \right\}\label{7-3}.
\end{align}
Lemma~\ref{lemma_background}  shows that the right hand side of  \eqref{7-3} is equal to 
\[
 \inf_{\gamma_i > -1} ~\frac{\gamma_i}{G_{ii}} \left(
\nu_i - \frac{\mu_i}{1 + \gamma_i}\right) ~=~ -\frac{1}{G_{ii}}(\sqrt{\mu_i}-\sqrt{\nu_i})^2.
\]
Without any restriction of generality we can assume that  
\be\label{7-1}
\min_{1\leq i\leq d}
~\frac{1}{G_{ii}}(\sqrt{\mu_i}-\sqrt{\nu_i})^2  ~ = ~\frac{1}{G_{11}}(\sqrt{\mu_1}-\sqrt{\nu_1})^2.
\ee
To get \eqref{7-2}, it is now sufficient to show that \eqref{7-3} holds
with the equality for $i=1$. 

For a given $\alpha\in\R^d$, it is convenient to introduce the set  $J(\alpha)$ of all those $j\in\{1,\ldots,d\}$ for which 
\[
e^{\alpha^k} = \, \sum_{j=1}^d p_{kj}e^{\alpha^j} +
p_{k0} . 
\]
The proof of equality  in  \eqref{7-3}  for $i=1$ uses the the following lemma.

\begin{lemma}\label{lem5-1} Suppose that the conditions  (A) and (B) are satisfied and let \eqref{7-1} hold. Suppose moreover that 
  $\alpha\in{\cal B}_{\{1\}}$ and   $\{2,\ldots,d\}\setminus J(\alpha)\not= \emptyset$. Then  for any $i\in\{2,\ldots,d\}\setminus J(\alpha)$,
there exists an $\tilde\alpha\in{\cal B}_{\{1\}}$ such that $J(\alpha)\cup\{i\}\subset J(\tilde\alpha)$
and 
\[
R(\tilde\alpha) ~<~ \max\left\{R(\alpha), ~ - ~\frac{1}{G_{ii}}\left(\sqrt{\mu_i}-\sqrt{\nu_i}\right)^2\right\}
\]
\end{lemma}
\begin{proof} Indeed, consider  the vector $\overrightarrow {\gamma_i}=(\gamma^1_i,\ldots,\gamma^d_i)$  defined by \eqref{2-1} with $\gamma_i = \gamma_i ^*=  \sqrt{\mu_i/\nu_i} - 1~>~ 0$. Then  for $k\not= i$,  using Lemma~\ref{lemma_background},
\begin{align}
\sum_{j=1}^d p_{kj}e^{\gamma^j_i} +
p_{i0}    ~=~ e^{\gamma_i^k}, \label{6-7}
\end{align}
from which it follows that  $\overrightarrow {\gamma_i} \in{\cal B}_{\{ i\}} \subset {\cal B}_{\{1, i\}}$. Moreover,
\begin{align}
\sum_{j=1}^d p_{ij}e^{\gamma^j_i} +
p_{i0}   &~=~   \, \sum_{j=1}^d p_{ij} (1 + Q_{ji} \gamma_i^*) +
p_{i0} ~=~   1 +  \sum_{j=1}^d p_{ij} Q_{ji} \gamma_i^*  \nonumber\\   &~=~  1 + \left(1 - \frac{1}{G_{ii}}\right) \gamma_i^*~<~  1 +  \gamma_i^* ~=~ e^{\gamma_i^i} \label{6-6}
\end{align} 
and
\be\label{6-5}
R(\overrightarrow {\gamma_i}) ~=~ \frac{\gamma_i^*}{G_{ii}} \left(
\nu_i - \frac{\mu_i}{1 + \gamma_i^*}\right) ~=~ - ~\frac{1}{G_{ii}}\left(\sqrt{\mu_i}-\sqrt{\nu_i}\right)^2,
\ee
Consider now the homeomorphism $\beta\to\alpha_{\{1,i\}}(\beta)$, from
  $\R_+^{\{1, i\}, d}$ to ${\cal B}_{\{1, i\}}$, defined by  \eqref{solution} for $\Lambda = \{1,i\}$,  and let $\alpha \to \beta_{\{1, i\}}(\alpha)$ denote
its inverse mapping. Then 
the equality \eqref{6-7}
implies that $\beta^k_{\{1, i\}}(\overrightarrow {\gamma_i}) = 0$ for all $k\in\{1,\ldots,d\}\setminus\{1, i\}$.
Suppose now that $\alpha\in{\cal B}_{\{1\}}$ and $i\not\in J(\alpha)$. Then  according to the definition of the set  ${\cal B}_{\{1\}}$, 
\[
\, \sum_{j=1}^d p_{ij}e^{\alpha^j} +
p_{i0} ~>~ e^{\alpha^i} 
\]
and $\alpha\in{\cal B}_{\{1, i\}}$. Since the function $R(\alpha_{\{1,i\}}(\beta))$ is continuous,  the last relation combined with \eqref{6-6} shows that for  some $0 < s < 1$,  the point  
$\tilde\beta =  s \beta_{\{1, i\}}(\overrightarrow {\gamma_i}) + (1-s) \beta_{\{1, i\}}(\alpha)\in
\R_+^{\{1, i\}, d}$ satisfies the equality 
\be\label{7-8}
\sum_{j=1}^d p_{ij}e^{\alpha^j_{\{1, i\}}(\tilde\beta)} +
p_{i0}  ~=~ e^{\alpha^i_{\{1, i\}}(\tilde\beta)}
\ee
and consequently, $i\in J(\alpha_{\{1, i\}}(\tilde\beta))$. Moreover,  $\tilde\beta_j=0$ for all those $j\in\{1,\ldots,d\}\setminus\{ i\}$
for which $\beta^j_{\{1, i\}}(\alpha) = 0$ and consequently, $J(\alpha) \subset
J(\alpha_{\{1, i\}}(\tilde\beta)) $.  Finally,  recall that by Proposition~\ref{prop_background}, the function $R(\alpha_{\{1,i\}}(\beta))$ is strictly convex. Hence, 
\[
R(\alpha_{\{1,i\}}(\tilde\beta)) ~<~ \max\{R(\alpha), R(\overrightarrow {\gamma_i})\},
\]
and therefore, our lemma is verified with $\tilde\alpha = \alpha_{\{1, i\}}(\tilde\beta)$. 
\end{proof}

\bigskip
Now we are ready to complete the proof of Theorem~\ref{theorem_3}. By induction with respect to the
set $J(\alpha)$, for any $\alpha\in{\cal B}_{\{1\}}$ with $J(\alpha)
\nsupseteq \{2,\ldots,d\}$ there is a point $\tilde\alpha\in {\cal B}_{\{1\}}$
with $J(\tilde\alpha)  \supseteq \{2,\ldots,d\}$
such that 
\[
R(\tilde\alpha) ~<~ \max\left\{R(\alpha), - \min_{2 \le i \le d}
~\frac{1}{G_{ii}}\left(\sqrt{\mu_i}-\sqrt{\nu_i}\right)^2\right\}.
\]
When combined with \eqref{7-1} and \eqref{6-5} for $i=1$, the last inequality shows that 
\[
R(\tilde\alpha) ~<~ \max\left\{R(\alpha), - 
~\frac{1}{G_{11}}\left(\sqrt{\mu_1}-\sqrt{\nu_1}\right)^2\right\} ~=~ \max\left\{R(\alpha), R(\overrightarrow {\gamma_1})\right\},  
\]
where, as in the proof of the last lemma,  $\overrightarrow {\gamma_1} =(\gamma^1_1,\ldots,\gamma^d_1)$  is defined by \eqref{2-1} with $\gamma_1 = \gamma_1 ^*=  \sqrt{\mu_1/\nu_1} - 1$. 
Since $J(\tilde\alpha)  = J(\overrightarrow {\gamma_1}) = \{2,\ldots,d\}$ and the minimum of  $R(\alpha)$ over $\alpha\in\R^d$ with $J(\alpha) = \{2,\ldots,d\}$ is achieved at the point $\overrightarrow {\gamma_1}$, using the last inequality we conclude that 
\[
R(\overrightarrow {\gamma_1}) ~\leq~  R(\tilde\alpha) ~<~ R(\alpha).
\]
This proves that the minimum of $R(\alpha)$ over $\alpha\in{\cal B}_{\{1\}}$
is achieved at  $\alpha = \overrightarrow {\gamma_1}$  and 
consequently,  equality holds in \eqref{7-3}  for $i=1$. The proof of Theorem~\ref{theorem_3} is complete.

\bibliographystyle{amsplain}
\bibliography{ref}
\end{document}